\documentclass[aop,preprint]{imsart}

\RequirePackage{amsthm,amsmath,amsfonts,amssymb,mathrsfs}
\RequirePackage[numbers,sort&compress]{natbib}
\RequirePackage[colorlinks,linkcolor=blue,citecolor=blue,urlcolor=blue]{hyperref}
\RequirePackage{graphicx}
\RequirePackage{aliascnt}
\RequirePackage[nameinlink,noabbrev]{cleveref}
\RequirePackage{enumitem,needspace}
\RequirePackage{booktabs,tabularx,array,multirow}

\startlocaldefs
\theoremstyle{plain}

\newtheorem{theorem}{Theorem}[section]
\newaliascnt{lemma}{theorem}
\newtheorem{lemma}[lemma]{Lemma}
\aliascntresetthe{lemma}

\newaliascnt{proposition}{theorem}
\newtheorem{proposition}[proposition]{Proposition}
\aliascntresetthe{proposition}

\newaliascnt{corollary}{theorem}

\aliascntresetthe{corollary}

\theoremstyle{definition}
\newaliascnt{assumption}{theorem}
\newtheorem{assumption}[assumption]{Assumption}
\aliascntresetthe{assumption}

\newtheorem{remark}{Remark}[section]

\newaliascnt{convention}{theorem}

\aliascntresetthe{convention}

\endlocaldefs

\begin{document}

\begin{frontmatter}
\title{Three quantitative central limit theorems for the unbounded excursion component of a Gaussian field}
\runtitle{Three quantitative central limit theorems}

\begin{aug}
\author[A]{\fnms{Yu}~\snm{Cheng} \ead[label=e1]{12531173@mail.sustech.edu.cn}},
\author[B]{\fnms{Song-hao}~\snm{Liu}\ead[label=e2]{liusonghao@dlut.edu.cn}},
\author[A,C]{\fnms{Qi-man}~\snm{Shao}\ead[label=e3]{shaoqm@sustech.edu.cn}}
\and
\author[A]{\fnms{Jing-yu}~\snm{Xu}\ead[label=e4]{12131253@mail.sustech.edu.cn}}


\address[A]{Department of Statistics and Data Science, Southern University of Science and Technology, Shenzhen, Guangdong, 518055, China \printead[presep={,\ }]{e3,e4,e1}}

\address[B]{School of Mathematical Sciences, Dalian University of Technology, Dalian, Liaoning, China \printead[presep={,\ }]{e2}}

\address[C]{Department of Statistics and Data Science, Shenzhen International Center for Mathematics, Southern University of Science and Technology, Shenzhen, Guangdong, 518055, China \printead[presep={,\ }]{e3}}
\end{aug}


\begin{abstract}
For a class of smooth stationary Gaussian fields, we study the
fluctuations of the volume, surface area and Euler characteristic
of the unbounded excursion component in growing boxes
$\Lambda_n=[-n,n]^d$.
We establish nondegenerate central limit theorems for all three
functionals throughout the supercritical regime in every
dimension $d\geq 2$.

We further quantify these central limit theorems by establishing
Berry--Esseen bounds of order $O\bigl(|\Lambda_n|^{-(1-1/s)}\bigr)$ in Kolmogorov distance, where $s\in(1,2]$ depend on the regularity
and decay assumptions.
Especially for the Bargmann--Fock field,
we obtain a Berry--Esseen bound of
$O\bigl(|\Lambda_n|^{-1/2}\bigr)$ for all three functionals.
A matching lower bound shows that this rate is sharp for the
Euler characteristic.
For the volume, we improve the previously known rate of
$O\bigl(|\Lambda_n|^{-1/2}(\log|\Lambda_n|)^{2d}\bigr)$.
We mainly use Stein's method and moment bounds for differences arising from white-noise resampling to finish our proof.
\end{abstract}



\end{frontmatter}

\section{Introduction}
\subsection{Background}\label{intro:background}
Percolation is a fundamental framework to study connectivity and phase transitions in random systems, see \citet{Grimmett1999},
for example the sharpness of phase transition
and the geometric structure of clusters in the critical and supercritical regimes, see \citet{DuminilCopinTassion2016,BollobasRiordan2006}.

In addition, continuum models of percolation have also attracted considerable interest, 
particularly in the connectivity of excursion sets of Gaussian random fields.
For a smooth, centred, stationary Gaussian field $f:\mathbb R^d\to\mathbb R$, $d\geq2$, its upper excursion set at level $\ell\in\mathbb R$ is defined by
\begin{align*}
    \{f\geq\ell\}:=\{x\in\mathbb R^d:f(x)\geq\ell\}.
\end{align*}

Early percolation criterion given by \citet{MolchanovStepanov1983b} shows that, 
under proper assumption on $f$, there exists $\ell_c\in\mathbb{R}$, such that 
\begin{align*}
\mathbb{P}\bigl(\{f\geq\ell\}\text{ has an unbounded component}\bigr)=
    \begin{cases}
        1, & \ell<\ell_c,\\
        0, & \ell>\ell_c.
    \end{cases}
\end{align*}

More recently, for $d=2$, the large-scale connectivity of Gaussian excursion sets is studied in \citet{BeffaraGayet2017,BeliaevMuirhead2018,RiveraVanneuville2019}, 
followed by sharpness of the phase transitions and identification that $\ell_c=0$ under different conditions,
see \citet{RiveraVanneuville2020,MuirheadVanneuville2020,MuirheadRiveraVanneuvilleKohlerSchindler2023}.
Further, sharpness results are extended to arbitrary dimensions under suitable assumptions in \citet{Severo2022}, 
while for the Bargmann--Fock field and $d\geq3$, \citet{DuminilCopinRiveraRodriguezVanneuville2023} proves that $\ell_c>0$.
Related work has further studied the critical behaviour of these models and the large-scale connectivity and geometry of excursion clusters, 
see \citet{GarbanVanneuville2020,MuirheadSevero2024,Muirhead2024,Severo2024,Vernotte2025,Vernotte2026}.
For related physical background and surveys, see \citet{Isichenko1992,Beliaev2023}.

\subsection{Limit behaviours of excursion sets}
A natural question is then to study the limiting behaviour of geometric and topological statistics of the excursion and level sets of $f$.
Under suitable assumptions, law of large numbers, central limit theorems and related asymptotic results are established for various such statistics;
see \citet{KratzLeon2001,Spodarev2013,NazarovSodin2016,MeschenmoserShashkin2011,MeschenmoserShashkin2013,Shashkin2013,BeliaevMcAuleyMuirhead2020,BeliaevMcAuleyMuirhead2024,HirschLachiezeRey2024}.
In particular, central limit theorems for the volume, surface area and Euler characteristic of the excursion set $\{f\geq\ell\}$ in growing observation windows 
are proved in \citet{BulinskiSpodarevTimmermann2012,Pham2013,EstradeLeon2016,Muller2017,KratzVadlamani2018}.

In particular, under suitable regularity, decay and positivity assumptions,
\citet{McAuley2026a} established laws of large numbers for these three functionals of the unbounded excursion component in growing box $\Lambda_n=[-n,n]^d$,
when $\ell<\ell_c$.
Moreover, nondegenerate central limit theorem was proved under the additional assumptions $q\geq0$ and $\ell<-\ell_c$, with $d=2$ for the Euler characteristic
\footnote{When $d=2$, $\ell_c=0$, so the conditions $\ell<-\ell_c$ and $\ell<\ell_c$ coincide.
Here $q$ is the kernel function of $f$, see \eqref{Gaussian field}.}.

In a subsequent work, \citet{McAuley2026b} obtained Berry--Esseen bounds for the volume at levels $\ell<\ell_c$, 
under suitable regularity, non-negative covariance and superpolynomial decay of truncated connection probability assumptions.
Explicitly, let $K(x-y)=\mathbb{E}[f(x)f(y)]$ be the covariance function of the field, and assume that 
\begin{align*}
    |\partial^\alpha K(x)|\leq C\left|x\right|^{-\beta},\quad \left|\alpha\right|\leq 1,
\end{align*}
then the normal approximation rate is $O\left(n^{-\frac d2+\frac{d^2}{\beta+d}+\varepsilon}\right)$, where $\epsilon>0$ is arbitrary.
For the Bargmann--Fock field, the rate is $O\left(n^{-d/2}(\log n)^{2d}\right)$.

Naturally, we may ask whether the nondegenerate central limit theorem holds at any supercritical level for surface area and Euler characteristic.
Further question is to derive volume-order Berry-Esseen bound for surface area and Euler characteristic functionals,
and to get a better convergence rate in volume case.
This is the motivation of this paper.

\subsection{Our contributions}
Within the Gaussian field framework of \citet{McAuley2026a},
we establish Berry--Esseen bounds for the volume, surface area and Euler characteristic of the unbounded excursion component in growing boxes.
Our approach combines an exchangeable-pair construction with Stein's method and moment estimates for white-noise resampling differences.

To make it clear, under Assumption~\ref{ass:McAuley}, with the kernel function $q\geq0$ and the respective decay conditions, 
we prove that the variances of the surface area and Euler characteristic are $\Theta(n^d)$ for every fixed $\ell<\ell_c$ and every $d\geq2$.
Consequently, nondegenerate central limit theorem holds for any supercritical level $\ell<-\ell_c$ and dimension $d$.

For each of the three functionals, we obtain normal approximation rate of order $O(n^{-d(1-1/s)})$, 
where $1<s\leq2$ depends on the smoothness and decay assumptions of the field.
Under sufficiently strong smoothness and decay conditions, taking $s=2$ gives the rate $O(n^{-d/2})$.
For the Euler characteristic, we also prove a lower bound of order $n^{-d/2}$, showing that the rate obtained at $s=2$ is sharp.
In the Bargmann--Fock field case, our volume bound removes the logarithmic factor in the result of \citet{McAuley2026b}.

Table~\ref{tab:comparison-McAuley} compares our results with the corresponding results of McAuley.
Only the restrictions on the level and dimension are displayed.
All quantitative rates are measured in Kolmogorov distance.

\begin{table*}[t]
    \centering
    \caption{}
    \label{tab:comparison-McAuley}
    \renewcommand{\arraystretch}{1.35}
    \setlength{\tabcolsep}{3pt}
    \scriptsize
    \begin{tabularx}{\textwidth}{
        @{}
        >{\centering\arraybackslash}p{0.12\textwidth}
        >{\raggedright\arraybackslash}p{0.19\textwidth}
        >{\centering\arraybackslash}p{0.14\textwidth}
        >{\raggedright\arraybackslash}X
        >{\raggedright\arraybackslash}p{0.17\textwidth}
        @{}
    }
        \toprule
        Functional
        & Reference
        & Level and dimension
        & CLT and Berry--Esseen bounds
        & Bargmann--Fock field
        \\
        \midrule

        \multirow{2}{0.12\textwidth}{\centering Volume}
        & \cite[Theorems~1.4 and~1.5]{McAuley2026a};
          \cite[Theorem~2.8]{McAuleyQuasi2026}
        & $\ell<\ell_c$, $d\geq2$
        & Nondegenerate CLT;
          $\displaystyle O\!\left(n^{-\frac d2+\frac{d^2}{\beta+d}+\varepsilon}\right)$,
          $\beta>d$, $\varepsilon>0$
        & $\displaystyle O\!\left(n^{-\frac d2}\log^{2d}n\right)$
        \\

        & Theorem~\ref{thm:volume}
        & $\ell<\ell_c$, $d\geq2$
        & $\displaystyle O\!\left(n^{-d(1-\frac1s)}\right)$, $1<s\leq2$
        & $\displaystyle O\!\left(n^{-\frac d2}\right)$
        \\

        \midrule

        \multirow{2}{0.12\textwidth}{\centering Surface area}
        & \cite[Theorems~1.4 and~1.5]{McAuley2026a}
        & $\ell<-\ell_c$, $d\geq2$
        & Nondegenerate CLT; no quantitative rate
        & No quantitative rate
        \\

        & Theorem~\ref{thm:surface}
        & $\ell<\ell_c$, $d\geq2$
        & $\displaystyle O\!\left(n^{-d(1-\frac1s)}\right)$, $1<s\leq2$
        & $\displaystyle O\!\left(n^{-\frac d2}\right)$
        \\

        \midrule

        \multirow{2}{0.12\textwidth}{\centering Euler characteristic}
        & \cite[Theorems~1.4 and~1.5]{McAuley2026a}
        & $\ell<-\ell_c$, $d=2$
        & Nondegenerate CLT; no quantitative rate
        & No quantitative rate
        \\

        & Theorem~\ref{thm:euler}
        & $\ell<\ell_c$, $d\geq2$
        & $\displaystyle c n^{-\frac d2}\leq d_{\mathrm K}\leq Cn^{-d(1-\frac1s)}$, $1<s\leq2$
        & $\displaystyle \Theta\!\left(n^{-\frac d2}\right)$
        \\

        \bottomrule
    \end{tabularx}

    \vspace{1mm}
    \begin{minipage}{0.98\textwidth}
        \footnotesize
        Table 1: Comparison with the results of McAuley \cite{McAuley2026a,McAuleyQuasi2026} for three functionals of the unbounded excursion component.
        Here $d_K$ is the Kolmogorov distance between target random variable and standard normal distribution,
        $\beta>d, 1<s\leq 2$ are decay parameters depend only on the field, other technical assumption for the results are omitted here.
    \end{minipage}
\end{table*}

This paper is organised as follows.
Section~\ref{sec:main-results} states the main results and the assumptions on the Gaussian field.
Section~\ref{sec:mixed-moment-tools} gives the auxiliary lemmas.
The main theorems are proved in Section~\ref{sec:geometric-normal}, and the auxiliary lemmas in Section~\ref{sec:lemma-proofs}.
Finally, Section~\ref{sec:variance-positivity} discuss about the limiting variances for surface area and Euler characteristic.

\section{Main Results}\label{sec:main-results}
In this section, we introduce basic assumptions on Gaussian fields, then state Berry--Esseen bound for three Minkowski functionals via exchangeable pair method. 

\subsection{A general Berry--Esseen bound}
Let $I$ be a finite or countable index set and $X=(X_i)_{i\in I}$ be a family of independent random elements.
For every $i\in I$, let $X_i'$ be an independent copy of $X_i$, with all copies mutually independent and independent of $X$.
If $A\subset I$ is finite, $X^A$ denotes the configuration obtained by replacing $X_i$ by $X_i'$ for all $i\in A$.
For a measurable functional $H=H(X)$, the resampling differences are defined by
\begin{align}
    D_iH&:=H(X)-H(X^{\{i\}}),\label{eq:Di}\\
    D_iD_jH&:=H(X)-H(X^{\{i\}})-H(X^{\{j\}})+H(X^{\{i,j\}}),\qquad i\ne j.\label{eq:Dij}
\end{align}

For $p\geq1$, write $\Vert Y\Vert_p=(\mathbb{E}|Y|^p)^{1/p}$ for the $L^p$ norm, and let
\begin{align*}
    d_{\mathrm K}(Y,Z):=\sup_{z\in\mathbb{R}}\left\lvert \mathbb{P}(Y\leq z)-\mathbb{P}(Z\leq z)\right\rvert
\end{align*}
be the Kolmogorov distance between $Y$ and $Z$.
Throughout, $Z$ denotes a standard normal random variable and its distribution function is denoted by $\Phi$.

First we state a general Berry--Esseen bound for $H$.

\begin{theorem}\label{thm:abstract}
    Let $1<s\leq2$ and $2s\leq p<\infty$, and let
    \begin{equation}\label{eq:pstar-conjugacy}
        p^*=\frac{sp}{p-s},\qquad \frac{1}{p}+\frac{1}{p^*}=\frac{1}{s}.
    \end{equation}

    Let $H=H(X)\in L^p$ and $\sigma^2=\operatorname{Var} (H)>0$.  
    Suppose there exist nonnegative arrays $(a_i)_{i\in I}$, $(c_i)_{i\in I}$, and $(b_{ij})_{i,j\in I}$ such that
    \begin{equation}\label{eq:ab-assumption}
        \Vert D_iH\Vert_p\leq a_i,\qquad \Vert D_iH\Vert_{p^*}\leq c_i,
        \qquad \Vert D_iD_jH\Vert_{p^*}\leq b_{ij}\quad(i\ne j),\qquad b_{ii}=c_i,
    \end{equation}
    and
    \begin{equation}\label{eq:Q-def-main}
        \mathcal Q_s(a,b):=\left\{\sum_{j\in I}\left(\sum_{i\in I}a_i b_{ij}\right)^{s}\right\}^{1/s}<\infty.
    \end{equation}

    Then, for finite or countable $I$,
    \begin{equation}\label{eq:abstract-main-bound}
        d_{\mathrm K}\!\left(\frac{H-\mathbb{E}[H]}{\sigma},Z\right)
        \leq C_s\frac{\mathcal Q_s(a,b)}{\sigma^2},\qquad Z\sim N(0,1).
    \end{equation}
Here $C_s$ depends only on $s$.
\end{theorem}

The proof of Theorem~\ref{thm:abstract} is given in Subsection~\ref{subsec:proof-abstract}.

For finite $I$, we may simply take $a_i=\Vert D_iH\Vert_p$, $c_i=\Vert D_iH\Vert_{p^*}$, $b_{ij}=\Vert D_iD_jH\Vert_{p^*}$ for $i\ne j$, and $b_{ii}=c_i$.
However, for a countable $I$, especially in this paper, we will need some additional effort to find proper moment controls.

\subsection{Berry-Esseen theorem for three Minkowski functionals of Gaussian fields}
Let $d\geq 2$, $W$ is the white noise on $\mathbb{R}^d$, and $q\in L^2(\mathbb{R}^d)$ is a real-valued even kernel function with $\Vert q\Vert_{L^2(\mathbb{R}^d)}=1$.  
We consider the following Gaussian field
\begin{align}\label{Gaussian field}
    f(x)=q*W(x)=\int_{\mathbb{R}^d}q(x-y)W(dy).
\end{align}

Then $f$ is a centered, stationary Gaussian field with unit variance.  
Denote $\ell_c$ for the percolation critical level of $f$, 
and $\{f\geq\ell\}$ is supercritical when $\ell<\ell_c$. 
Under conditions (M1)--(M4) below, $\ell_c$ is finite and, for every $\ell<\ell_c$, the excursion set $\{f\geq\ell\}$ almost surely has a unique unbounded component; 
see \citet{MuirheadVanneuville2020}, \citet{Rivera2021}, and \citet{Severo2022,Severo2024}, also see \citet[Theorem~1.1]{McAuley2026a}.

In this paper we apply several basic assumptions from \citet{McAuley2026a}.

\begin{assumption}\label{ass:McAuley}
    There are an integer $k\geq 4$, a number $\epsilon_0>0$, an exponent $\beta>d$, and a constant $C<\infty$ such that:
    \begin{enumerate}[label=(M\arabic*)]
    \item The covariance $K=q*q$ belongs to $C^{2k+\epsilon_0}$;
    \item $q\in C^3(\mathbb{R}^d)$ and $\partial^\alpha q\in L^2(\mathbb{R}^d)$ for every multi-index $\alpha$ with $\left|\alpha\right|\leq 3$;
    \item For $\left|x\right|\geq 1$, and for every multi-index $\alpha$ with $\left|\alpha\right|\leq 2$,
    \begin{align*}
        |\partial^\alpha q(x)|\leq C\left|x\right|^{-\beta};
    \end{align*}
    \item $q$ is invariant under sign changes and coordinate permutations; moreover, $K=q*q\geq 0$ when $d=2$, while $q\geq 0$ when $d\geq 3$;
    \item Fix a level $\ell<\ell_c$, for every $a>0$,
    \begin{align*}
        n^a\mathbb{P}\!\left(\Lambda_1\text{ intersects a bounded component of }\{f\geq\ell\}\text{ of diameter greater than }n\right)\longrightarrow0.
    \end{align*}
    \end{enumerate}
\end{assumption}

Conditions (M1)--(M3) provide regularity, local moments, and spatial decay; (M4) supplies symmetry and positivity, as well as the known percolation structure. 
Condition (M5) is a technical condition which is related to the uniqueness of the unbounded component of $\{f\geq \ell\}$.
Part of these assumptions will not be applied in our proof directly.
Also, by \citet[Proposition~1.2]{McAuley2026a}, (M1)--(M4) imply (M5) when $\ell<-\ell_c$.
See \citet{McAuley2026a} and references therein for details.

For $\ell<\ell_c$, denote the unique unbounded component by $\{f\geq\ell\}_\infty$. 
Throughout, let $\Lambda_n=[-n,n]^d$ for $n\in\mathbb N$, 
and the observation sets are $\Lambda_n$, closed cubical faces, and finite unions of such faces. 
For any observation set $A\subset\mathbb{R}^d$, the three target functionals are defined as
\begin{align*}
    \mu_{\operatorname{Vol}}(A;f)&=\operatorname{Vol}(A\cap\{f\geq\ell\}_\infty),\\
    \mu_{\operatorname{SA}}(A;f)&=\mathcal H^{d-1}(A\cap\partial\{f\geq\ell\}_\infty),\\
    \mu_{\operatorname{EC}}(A;f)&=\operatorname{EC}(A\cap\{f\geq\ell\}_\infty).
\end{align*}

Here $\operatorname{Vol}$, $\mathcal H^{k}$, and $\operatorname{EC}$ stand for volume, $k$-dimensional Hausdorff measure, and Euler characteristic, respectively.
We abbreviate $\mu_\star(A)=\mu_\star(A;f)$. Here $\star\in\{\operatorname{Vol}, \operatorname{SA}, \operatorname{EC}\}$. 

Let $B_v=v+[0,1]^d, v\in\mathbb{Z}^d$, and write $X_v=W|_{B_v}$.  
Here the restriction means that for any test function $h\in L^2(\mathbb{R}^d)$,
\begin{align*}
    X_v(h)=W(\mathbf{1}_{B_v}h)=\int_{\mathbb{R}^d}\mathbf{1}_{B_v}h(x)W(dx).
\end{align*}

Then for any distinct $v,w\in\mathbb{Z}^d$ and any $h,k\in L^2(\mathbb R^d)$,
\begin{align*}
    \mbox{Cov}(X_v(h),X_w(k))=\int_{\mathbb{R}^d}\mathbf{1}_{B_v}\mathbf{1}_{B_w}h(x)k(x)dx=0,
\end{align*}
thus $X=(X_v)_{v\in\mathbb{Z}^d}$ are independent random elements.

Let $X_v'$ be an independent copy of $X_v$, and let $W^{(v)}$ be obtained from $W$ by replacing $X_v$ with $X_v'$. 
The resampled field and the perturbation are defined as
\begin{equation}\label{eq:block-perturbation}
    f^{(v)}=q*W^{(v)},\qquad p_v=f^{(v)}-f=q*(X_v'-X_v).
\end{equation}

For details of this resampling, see \citet[Section~1.3]{McAuley2026a}.

\begin{remark}[Measurability]\label{rem:field-measurability}
    For each fixed observation set $A\subset\mathbb{R}^d$, the three functionals $\mu_\star(A)$ admit versions measurable with respect to the white-noise blocks $X$; 
    see \citet[Section~1.3]{McAuley2026a}.
\end{remark}

Now we can state our main results on these functionals.
\begin{theorem}\label{thm:volume}
    Under Assumption~\ref{ass:McAuley}, fix $1<s\leq2$ and assume
    
    \begin{equation}\label{eq:beta-vol}
        \beta>3sd.
    \end{equation}
    
    There exist $C_{\operatorname{Vol}}>0$ and $n_0\in\mathbb N$ such that
    \begin{equation}\label{eq:vol-rate}
        d_{\mathrm K}\!\left(
        \frac{\mu_{\operatorname{Vol}}(\Lambda_n)-\mathbb{E}[\mu_{\operatorname{Vol}}(\Lambda_n)]}
        {\sqrt{\operatorname{Var}\left(\mu_{\operatorname{Vol}}(\Lambda_n)\right)}},Z\right)
        \leq C_{\operatorname{Vol}}n^{-d(1-1/s)},\qquad n\geq n_0.
    \end{equation}

\end{theorem}

\begin{theorem}\label{thm:surface}
    Under Assumption~\ref{ass:McAuley}, assume also $q\geq0$,
    fix $1<s\leq2$ and put
    \begin{equation}\label{eq:kprime}
        k'=\frac{k(k+1)}2-2.
    \end{equation}

    Assume
    \begin{equation}\label{eq:beta-SA}
        \beta>3sd\,\frac{3k'-1}{k'-2s}.
    \end{equation}

    There exist $C_{\operatorname{SA}}>0$ and $n_0\in\mathbb N$ such that
    \begin{equation}\label{eq:SA-rate}
        d_{\mathrm K}\!\left(
        \frac{\mu_{\operatorname{SA}}(\Lambda_n)-\mathbb{E}[\mu_{\operatorname{SA}}(\Lambda_n)]}
        {\sqrt{\operatorname{Var}\left(\mu_{\operatorname{SA}}(\Lambda_n)\right)}},Z\right)
        \leq C_{\operatorname{SA}}n^{-d(1-1/s)},\qquad n\geq n_0.
    \end{equation}

\end{theorem}

\begin{theorem}\label{thm:euler}
    Under Assumption~\ref{ass:McAuley}, assume also $q\geq0$,
    fix $1<s\leq2$ with $k-1>2s$, and suppose
    \begin{equation}\label{eq:beta-EC}
        \beta>3sd\,\frac{k-1}{k-1-2s}.
    \end{equation}
    There exist $0<c_{\operatorname{EC}}\leq C_{\operatorname{EC}}<\infty$ and $n_0\in\mathbb N$ such that, for all $n\geq n_0$,
    \begin{equation}\label{eq:EC-rate}
        c_{\operatorname{EC}}n^{-d/2}\leq
        d_{\mathrm K}\!\left(
        \frac{\mu_{\operatorname{EC}}(\Lambda_n)-\mathbb{E}[\mu_{\operatorname{EC}}(\Lambda_n)]}
        {\sqrt{\operatorname{Var}\left(\mu_{\operatorname{EC}}(\Lambda_n)\right)}},Z\right)
        \leq C_{\operatorname{EC}}n^{-d(1-1/s)}.
    \end{equation}

\end{theorem}

The proofs of Theorems~\ref{thm:volume}--\ref{thm:euler} are given in Subsection~\ref{subsec:geometric-applications}.

\begin{remark}
    The constants in (\ref{eq:vol-rate}), (\ref{eq:SA-rate}) and (\ref{eq:EC-rate}) only depend on the field, the level $\ell$, the dimension $d$, 
    and parameters in our assumptions. 
\end{remark}

\begin{remark}\label{rem:parameter-matching}
In Theorem~\ref{thm:volume}--\ref{thm:euler}, let $s\downarrow1$ and $s=2$, respectively. 
The spatial decay parameter $\beta$ in (M3) of Assumption~\ref{ass:McAuley} satisfies the following conditions.
\begin{center}
\renewcommand{\arraystretch}{1.5}
\begin{tabular}{@{}lcc@{}}
\hline
Functional & Limiting threshold as $s\downarrow1$ & $s=2$: rate $n^{-d/2}$\\
\hline
Volume, $k\geq4$ & $\beta>3d$ & $\beta>6d$\\
Surface area, $k\geq4$ & $\displaystyle\beta>\frac{3d(3k'-1)}{k'-2}$
                     & $\displaystyle\beta>\frac{6d(3k'-1)}{k'-4}$\\
Euler characteristic & $\displaystyle k\geq4,\ \beta>\frac{3d(k-1)}{k-3}$
                     & $\displaystyle k\geq6,\ \beta>\frac{6d(k-1)}{k-5}$\\
\hline
\end{tabular}
\end{center}
where $k'=k(k+1)/2-2$. 

Thus our decay conditions on $q$ correspond to those in \citet[Theorem~1.4]{McAuley2026a} when $s\downarrow1$.
At $s=2$, the rate is $n^{-d/2}$ for all these functionals, which is optimal in EC case.
\end{remark}

\begin{remark}\label{rem:variance}
    There is a tradeoff between the assumption of decay exponent $\beta$ and convergence rate in our Berry-Esseen bounds. 
    The Bargmann--Fock field satisfies the regularity, symmetry, positivity, and decay conditions above, 
    and Mat\'{e}rn fields do so when their smoothness parameter is sufficiently large, see Example~1 of \citet{McAuley2026a}.
\end{remark}

\begin{remark}\label{rem:vol-SA-positive}
Under Assumption~\ref{ass:McAuley} and the respective limit thresholds in Remark~\ref{rem:parameter-matching}, 
\citet[Theorem~1.4]{McAuley2026a} proved that
\begin{align*}
    \frac{\operatorname{Var}\mu_\star(\Lambda_n)}{|\Lambda_n|}
    \longrightarrow\sigma_\star(\ell)^2\in[0,\infty),
    \qquad \star\in\{\operatorname{Vol},\operatorname{SA},\operatorname{EC}\},
\end{align*}
where $|\Lambda_n|=(2n)^d$.  
In the same paper, \citet[Theorem~1.5]{McAuley2026a} proves $\sigma_\star(\ell)^2>0$ under the additional assumptions $q\geq0$ and $\ell<-\ell_c$ for volume and surface area, 
further assume $d=2$ for Euler characteristic.
Subsequently, \citet[Theorem~2.8]{McAuleyQuasi2026} established strict positivity for volume at every fixed $\ell<\ell_c$ under a similar condition to Assumption~\ref{ass:McAuley}.

For the other two functionals,
Proposition~\ref{prop:SA-positive} and Proposition~\ref{prop:EC-positive} below
establish strict positivity of $\sigma_\star(\ell)$ for any $\ell<\ell_c$ and $d\geq 2$ under Assumption~\ref{ass:McAuley}, $q\geq0$, and the
respective conditions on $\beta$.
\end{remark}

\begin{proposition}\label{prop:SA-positive}
    Under Assumption~\ref{ass:McAuley}, assume also $q\geq0$,
    let $k'=k(k+1)/2-2$ and suppose that
    \begin{equation}\label{eq:SA-positive-decay}
        \beta>\frac{3d(3k'-1)}{k'-2}.
    \end{equation}
    Then
    \begin{equation}\label{eq:SA-positive-limit}
        \frac{\operatorname{Var}\mu_{\operatorname{SA}}(\Lambda_n)}{(2n)^d}
        \longrightarrow\sigma_{\operatorname{SA}}(\ell)^2\in(0,\infty).
    \end{equation}
\end{proposition}

\begin{proposition}\label{prop:EC-positive}
    Under Assumption~\ref{ass:McAuley}, assume also $q\geq0$ and
    \begin{equation}\label{eq:EC-positive-decay}
        \beta>\frac{3d(k-1)}{k-3}.
    \end{equation}
    Then
    \begin{equation}\label{eq:EC-positive-limit}
        \frac{\operatorname{Var}\mu_{\operatorname{EC}}(\Lambda_n)}{(2n)^d}
        \longrightarrow\sigma_{\operatorname{EC}}(\ell)^2\in(0,\infty).
    \end{equation}
\end{proposition}

The proofs of Propositions~\ref{prop:SA-positive} and~\ref{prop:EC-positive} are given in
Subsection~\ref{subsec:SA-variance-positive} and Subsection~\ref{subsec:EC-variance-positive}, respectively.

\section{Auxiliary lemmas}\label{sec:mixed-moment-tools}
In this section, we state the technical lemmas used in the proofs of the main theorems.
Lemmas~\ref{lem:doob-identities}--\ref{lem:uniform-low-moment} provide moment estimates used in Theorem~\ref{thm:abstract};
Lemmas~\ref{lem:gen-pair} and~\ref{lem:local-sum} derive Berry--Esseen bounds;
Lemmas~\ref{lem:volume-envelope}--\ref{lem:EC-envelope} give moment estimates of resampling differences.
The proofs of these lemmas are given in Section~\ref{sec:lemma-proofs}.

\subsection{Exchangeable pair method}
Let $X=(X_i)_{i\in I}$ be a group of independent random elements.
First we consider $I=[m], m\in\mathbb{N}$.
Under the assumption of Theorem~\ref{thm:abstract}, replace $H$ by $H-\mathbb{E}H$.
This replacement does not change any resampling difference $D_iH$ or $D_iD_jH$.
Define
\begin{align*}
    \mathcal{F}_i=\sigma(X_1,\ldots,X_i),\qquad Y_i=\mathbb{E}[H\mid\mathcal{F}_i]-\mathbb{E}[H\mid\mathcal{F}_{i-1}].
\end{align*}

Choose a measurable version $Y_i=h_i(X_1,\ldots,X_i)$ and define
\begin{align*}
    Y_i^{(i)}=h_i(X_1,\ldots,X_{i-1},X_i'),\qquad D_iY_i=Y_i-Y_i^{(i)}.
\end{align*}

\begin{lemma}\label{lem:doob-identities}
    Let $J$ be uniform on $[m]$ and independent of all other variables.  
    For $\sigma^2=\operatorname{Var} H$, let
    \begin{equation}\label{eq:pair-construction}
        W=H/\sigma,\quad W'=H(X^{\{J\}})/\sigma,\quad D=D_JY_J/\sigma.
    \end{equation}

    Then
    \begin{equation}\label{eq:doob-joint-symmetry}
        (W,W',D,|D|)\stackrel{d}{=}(W',W,-D,|D|),
    \end{equation}
    and
    \begin{equation}\label{eq:regression-doob}
        \mathbb{E}[D\mid X]=\frac{1}{m}W.
    \end{equation}

    Moreover, for every $1\leq u\leq p$,
    \begin{equation}\label{eq:DiYi-contract}
        \Vert D_iY_i\Vert_{u}\leq\Vert D_iH\Vert_{u},
    \end{equation}
    and, for $j\ne i$,
    \begin{equation}\label{eq:DjDiYi-contract}
        \Vert D_jD_iY_i\Vert_{u}\leq\Vert D_jD_iH\Vert_{u}
        \quad\text{if }j<i,
        \qquad D_jD_iY_i=0\quad\text{if }j>i.
    \end{equation}

\end{lemma}

For each $i$, denote
\begin{align}\label{eq:Ai and Bi}
    A_i=\mathbb{E}[(D_iY_i)(D_iH)\mid X],\quad B_i=\mathbb{E}[|D_iY_i|(D_iH)\mid X].
\end{align}

Both $A_i=A_i(X)$ and $B_i=B_i(X)$ are functionals of $X$. 
Thus we can define resampling differences of $A_i$ and $B_i$ as
\begin{align*}
    D_jA_i=A_i(X)-A_i(X^{\{j\}}),\qquad D_jB_i=B_i(X)-B_i(X^{\{j\}}).
\end{align*}

\begin{lemma}\label{lem:proxy-differences}
    Let $I=[m]$.
    Suppose $a_i,c_i,b_{ij}$ dominate the corresponding norms as in \eqref{eq:ab-assumption}.  
    Then, for all $i,j\in I$,
    \begin{equation}\label{eq:proxy-diff}
        \Vert D_jA_i\Vert_{s}\leq2a_ib_{ij},
        \qquad
        \Vert D_jB_i\Vert_{s}\leq2a_ib_{ij}.
    \end{equation}

\end{lemma}

\begin{lemma}[von Bahr--Esseen inequality \cite{vonBahrEsseen1965}]\label{lem:uniform-low-moment}
    Let $X=(X_i)_{i\in I}$ be a family of independent random elements.
    If $1<s\leq2$ and $G=G(X)\in L^s$, then
    \begin{equation}\label{eq:uniform-low-moment}
        \Vert G-\mathbb EG\Vert_s \leq C_s\left(\sum_{j\in I}\Vert D_jG\Vert_s^s\right)^{1/s},
    \end{equation}
    where the index set $I$ is finite or countably infinite.    
\end{lemma}

The proofs of Lemmas~\ref{lem:doob-identities}--\ref{lem:uniform-low-moment} are given in Subsection~\ref{subsec:proof-exchangeable}.

Generally, let $(X,X')$ be an exchangeable pair on a measurable space $\mathcal X$, 
let $W=\phi(X)$, and $F:\mathcal X^2\to\mathbb R$ be antisymmetric, with $D=F(X,X')$.
Under these exchangeable-pair assumptions, \citet[Theorem~4.1]{Zhang2022} gives the following Berry--Esseen bound with zero remainder term.

\begin{lemma}\label{lem:gen-pair}
    Let $W'=\phi(X')$. Assume that there exists a square-integrable random variable $D^*\geq|D|$ such that
    \begin{align*}
        (W,W',D,D^*)\stackrel d=(W',W,-D,D^*).
    \end{align*}

    Denote $\Delta=W-W'$.

    Suppose $\mathbb{E}[W]=0$, $\mathbb{E}[W^2]=1$, 
    and that there is a sigma-field $\mathcal{G}$ such that $W\in\mathcal{G}$ and
    \begin{equation}\label{eq:pair-regression}
        \mathbb{E}[D\mid\mathcal G]=\lambda W,\qquad \lambda>0.
    \end{equation}

    Then
    \begin{align}\label{eq:gen-pair-bound}
        d_{\mathrm K}(W,Z)\leq \mathbb{E}\left|1-\frac1{2\lambda}\mathbb{E}[D\Delta\mid W]\right|+\frac{1}{\lambda}\mathbb{E}\left|\mathbb{E}[D^*\Delta\mid W]\right|.     
    \end{align}
\end{lemma}

\subsection{Berry--Esseen bound}
Now we apply Theorem~\ref{thm:abstract} when $H$ is a local-sum, which gives Berry--Esseen bound in Theorem~\ref{thm:volume}--\ref{thm:euler}.

Let $(X_v)_{v\in\mathbb{Z}^d}$ be independent random elements, with resampling differences defined in \eqref{eq:Di}--\eqref{eq:Dij}.
For $i\not=j\in\mathbb{Z}^d$ and a general functional $\xi$ of $X$, define 
\begin{align*}
    D_i\xi(X^{\{j\}})=\xi(X^{\{j\}})-\xi(X^{\{i,j\}})
\end{align*}

\begin{lemma}\label{lem:local-sum}
    Let $L\subset\mathbb{Z}^d$ be finite.  
    For each $w\in L$, let $\xi_w=\xi_w(X)$ be a measurable real-valued functional, and
    \begin{align}\label{partial-sum}
        H=\sum_{w\in L}\xi_w
    \end{align}
    be the local-sum.

    Fix $1<s\leq2$ and assume $H\in L^p$, where $p\geq2s$ and $p^*$ satisfy \eqref{eq:pstar-conjugacy}.  
    Suppose there are exponents $\theta_p,\theta_{p^*}>0$ and finite positive constants
    $C_p,C_{p^*}$ such that, for every $i\not=j\in\mathbb Z^d$ and $w\in L$,
    there exists a nonnegative random variable $Y_i^B(w)$ satisfying
    \begin{align}\label{eq:site-envelope}
        |D_i\xi_w(X^B)|&\leq Y_i^B(w),\notag\\
        \mathbb{E}\bigl(Y_i^B(w)\bigr)^p\leq C_p(1+|i-w|)^{-\theta_{p}}&\quad
        \mathbb{E}\bigl(Y_i^B(w)\bigr)^{p^*}\leq C_{p^*}(1+|i-w|)^{-\theta_{p^*}},
    \end{align}
    where $B=\emptyset$ or $\{j\}$.

    Further assume that
    \begin{align}\label{eq:mixed-summability-criterion}
        \frac{\theta_p}{p}>d,\qquad
        \frac{\theta_{p^*}}{p^*}>2d.
    \end{align}
    and denote
    \begin{align}\label{eq:site-kernel-sums}
        K_p=\sum_{r\in\mathbb{Z}^d}(1+|r|)^{-\theta_p/p},
        K_{p^*}=\sum_{r\in\mathbb{Z}^d}(1+|r|)^{-\theta_{p^*}/p^*},
        T_{p^*}=\sum_{r\in\mathbb{Z}^d}(1+|r|)^{-\theta_{p^*}/(2p^*)}.
    \end{align}

    Then there exist nonnegative arrays $(a_i)$, $(c_i)$, and $(b_{ij})$ satisfying \eqref{eq:ab-assumption} for $H$. 
    Additionally, for a finite constant
    $C=C(d,p,p^*,s,\theta_p,\theta_{p^*},C_p,C_{p^*})$,
    \begin{align}\label{eq:Q-local}
        \mathcal Q_s(a,b)\leq C|L|^{1/s}.
    \end{align}

    Consequently, whenever $\sigma^2:=\operatorname{Var} H>0$,
    \begin{align}\label{eq:local-general}
        d_{\mathrm K}\!\left(\frac{H-\mathbb{E} H}{\sigma},Z\right)
        \leq C\frac{|L|^{1/s}}{\sigma^2}.
    \end{align}

    In particular, if $\operatorname{Var} H\geq c_0|L|$, then
    \begin{align}\label{eq:local-rate}
        d_{\mathrm K}\!\left(\frac{H-\mathbb{E} H}{\sqrt{\operatorname{Var} H}},Z\right)
        \leq C(c_0,d,p,p^*,s,\theta_p,\theta_{p^*},C_p,C_{p^*})|L|^{-(1-1/s)}.
    \end{align}

\end{lemma}

The proof of Lemma~\ref{lem:local-sum} is given in Subsection~\ref{subsec:proof-spatial}.

\subsection{Moment estimates}
To apply Lemma~\ref{lem:local-sum} to the three target functionals $\mu_{\star}(\Lambda_n;f)$,
we write them in the form of \eqref{partial-sum} and find suitable controlling random variables.
Here we apply the controlling random variables constructed in \citet{McAuley2026a}, and derive higher-order moment estimates.

In the sequel denote $R(v,w)=1+\left|v-w\right|$, where $v,w\in\mathbb{Z}^d$. 

\begin{lemma}\label{lem:volume-envelope}
    Under Assumption~\ref{ass:McAuley}, for every $1\leq r<\infty$ and every $\eta>0$ there is $C_{r,\eta}<\infty$ such that, 
    for any $v,w\in\mathbb{Z}^d$,
    there exists $Y^{\operatorname{Vol}}_v(w)$ with
    \begin{align}\label{eq:vol-envelope}
        |D_v\mu_{\operatorname{Vol}}(B_w)|\leq Y^{\operatorname{Vol}}_v(w)\leq2
    \end{align}
    and
    \begin{align}\label{eq:vol-r-envelope}
        \mathbb{E}\bigl(Y^{\operatorname{Vol}}_v(w)\bigr)^r
        \leq C_{r,\eta} R(v,w)^{-\beta+\eta}.
    \end{align}

\end{lemma}

\begin{lemma}\label{lem:surface-envelope}
    Under Assumption~\ref{ass:McAuley}, let $k'=k(k+1)/2-2$.  
    For every $1\leq r<k'$ and every $\eta>0$, there is $C_{r,\eta}<\infty$ such that, 
    for any $v,w\in\mathbb{Z}^d$,
    there exists $Y^{\operatorname{SA}}_v(w)$ with
    \begin{align}\label{eq:SA-envelope}
        |D_v\mu_{\operatorname{SA}}(B_w)|\leq Y^{\operatorname{SA}}_v(w)
    \end{align}
    and
    \begin{align}\label{eq:SA-r-envelope}
        \mathbb{E}\bigl(Y^{\operatorname{SA}}_v(w)\bigr)^r
        \leq C_{r,\eta}R(v,w)^{-\beta\frac{k'-r}{3k'-1}+\eta}.
    \end{align}

\end{lemma}

Under Assumption~\ref{ass:McAuley}, for finite unions $P,Q$ of closed cubical faces,
\citet[Theorem~6.1.1]{AdlerTaylor2007} gives, almost surely,
\begin{equation}\label{eq:EC-valuation}
    \mu_{\operatorname{EC}}(P\cup Q;f)=\mu_{\operatorname{EC}}(P;f)+\mu_{\operatorname{EC}}(Q;f)-\mu_{\operatorname{EC}}(P\cap Q;f).
\end{equation}

For an integer $n\geq1$, set
\begin{equation}\label{eq:EC-cube-geometry}
    I_n=\{-n,\ldots,n-1\}^d,\qquad \Gamma_{n,w}=\bigcup_{\substack{j\in[d]\colon w_j<n-1}}\bigl(B_w\cap\{x_j=w_j+1\}\bigr),\quad w\in I_n,
\end{equation}
where $\Gamma_{n,w}$ is the union of the upper sides of $B_w$ shared with other cubes in $\Lambda_n$. 
Then define
\begin{equation}\label{eq:EC-cube-score}
    \xi_w(f)=\mu_{\operatorname{EC}}(B_w;f)-\mu_{\operatorname{EC}}(\Gamma_{n,w};f).
\end{equation}

\begin{lemma}\label{lem:EC-decomp}
    Under Assumption~\ref{ass:McAuley}, almost surely, for every integer $n\geq1$,
    \begin{equation}\label{eq:EC-exact-sum}
    \mu_{\operatorname{EC}}(\Lambda_n;f)=\sum_{w\in I_n}\xi_w(f).
    \end{equation}

    Further, for distinct $u,v\in\mathbb Z^d$,
    \begin{align}
        D_v\mu_{\operatorname{EC}}(\Lambda_n;f)&=\sum_{w\in I_n}D_v\xi_w(f),\label{eq:EC-first-difference-sum}\\
        D_uD_v\mu_{\operatorname{EC}}(\Lambda_n;f)&=\sum_{w\in I_n}D_uD_v\xi_w(f).\label{eq:EC-second-difference-sum}
    \end{align}
\end{lemma}

\begin{lemma}\label{lem:EC-envelope}
    Under Assumption~\ref{ass:McAuley}, functionals $\xi_w$ defined in \eqref{eq:EC-cube-score} satisfy
    \begin{equation}\label{eq:EC-score-moment}
        \sup_{n\geq1}\sup_{w\in I_n}\mathbb E|\xi_w(f)|^{k-1}<\infty.
    \end{equation}

    For every $1\leq r<k-1$ and every $\eta>0$, 
    there is $C_{r,\eta}<\infty$ such that, for any $v,w\in\mathbb Z^d$,
    there exists a nonnegative random variable $Y_v^{\operatorname{EC}}(w)$, such that

    \begin{equation}\label{eq:EC-local-envelope}
        |D_v\xi_w(f)|\leq Y_v^{\operatorname{EC}}(w),\qquad n\geq1,\quad w\in I_n,
    \end{equation}
    and
    \begin{equation}\label{eq:EC-r-envelope}
        \mathbb E\bigl(Y_v^{\operatorname{EC}}(w)\bigr)^r\leq C_{r,\eta}R(v,w)^{-\beta\frac{k-1-r}{k-1}+\eta}.
    \end{equation}
\end{lemma}

The proofs of Lemmas~\ref{lem:volume-envelope}--\ref{lem:EC-envelope} are given in Subsection~\ref{subsec:proof-geometric}.

\section{Proofs of the main theorems}\label{sec:geometric-normal}
\subsection{Proof of Theorem~\ref{thm:abstract}}\label{subsec:proof-abstract}
\begin{proof}
    We first assume that $I=[m]$ and replace $H$ by $H-\mathbb{E} H$.  
    Construct $(W,W',D)$ as in \eqref{eq:pair-construction}.  
    Since $p\geq2s>2$, $W,W'$ and $D$ are square-integrable. 
    
    \eqref{eq:doob-joint-symmetry} and \eqref{eq:regression-doob} in Lemma~\ref{lem:doob-identities} verify the joint symmetry and Stein pair condition in Lemma~\ref{lem:gen-pair}, 
    with $\mathcal G=\sigma(X)$, $\lambda=1/m$, and $D^*=|D|$.
    Denote $\Delta=W-W'$, as defined in Lemma~\ref{lem:gen-pair}.
    Recall from \eqref{eq:Ai and Bi} that
    \begin{align*}
        A_i=\mathbb{E}[(D_iY_i)(D_iH)\mid X],\quad B_i=\mathbb{E}[|D_iY_i|(D_iH)\mid X],
    \end{align*}

    By definition and $\lambda=1/m$,
    \begin{equation}\label{eq:pair-proxy-sums}
    \begin{aligned}
        \frac1{2\lambda}\mathbb E[D\Delta\mid X]
        &=\frac{m}{2\sigma^2}\mathbb E[D_JY_J\,D_JH\mid X]
        =\frac1{2\sigma^2}\sum_{i=1}^m\mathbb E[D_iY_i\,D_iH\mid X]=\frac1{2\sigma^2}\sum_{i=1}^m A_i,\\
        \frac1\lambda\mathbb E[D^*\Delta\mid X]
        &=\frac{m}{\sigma^2}\mathbb E[|D_JY_J|\,D_JH\mid X]
        =\frac1{\sigma^2}\sum_{i=1}^m\mathbb E[|D_iY_i|\,D_iH\mid X]=\frac1{\sigma^2}\sum_{i=1}^m B_i.
    \end{aligned}
    \end{equation}

    Since $D_iY_i=\mathbb E[D_iH\mid X_1,\ldots,X_i,X_i']$,
    conditional expectation and orthogonality of the martingale differences give
    \begin{equation}\label{eq:proxy-mean-sums}
    \begin{aligned}
        \sum_{i=1}^m\mathbb EA_i
        =\sum_{i=1}^m\mathbb E[(D_iY_i)(D_iH)]&=\sum_{i=1}^m\mathbb E[(D_iY_i)^2]=2\sum_{i=1}^m\mathbb E[Y_i^2]=2\sigma^2,\\
        \sum_{i=1}^m\mathbb EB_i
        &=\sum_{i=1}^m\mathbb E[|D_iY_i|(D_iH)]=0.
    \end{aligned}
    \end{equation}
    Here $Y_i$ and $Y_i^{(i)}$ are conditionally independent copies with conditional mean zero given $\mathcal F_{i-1}$.
    Interchanging $X_i$ and $X_i'$ leaves $|D_iY_i|$ unchanged and reverses the sign of $D_iH$, giving $\mathbb EB_i=0$.

    Since $W$ is measurable with respect to $X$, Lemma~\ref{lem:gen-pair} and
    \eqref{eq:pair-proxy-sums}--\eqref{eq:proxy-mean-sums} yield
    \begin{equation}\label{eq:DK-proxy-sums}
    \begin{aligned}
        d_{\mathrm K}(W,Z)
        &\leq\frac1{2\sigma^2}\mathbb E\left|\sum_{i=1}^m A_i-2\sigma^2\right|
             +\frac1{\sigma^2}\mathbb E\left|\sum_{i=1}^m B_i\right|\\
        &\leq\frac1{2\sigma^2}\left\Vert\sum_{i=1}^m(A_i-\mathbb EA_i)\right\Vert_s
             +\frac1{\sigma^2}\left\Vert\sum_{i=1}^m(B_i-\mathbb EB_i)\right\Vert_s.
    \end{aligned}
    \end{equation}
    Apply Lemmas~\ref{lem:uniform-low-moment} to $\sum_i A_i$ and $\sum_i B_i$, 
    then apply Lemma~\ref{lem:proxy-differences} and Minkowski's inequality, we have
    \begin{align*}
        \left\Vert\sum_i A_i-\mathbb E\sum_i A_i\right\Vert_s+\left\Vert\sum_i B_i-\mathbb E\sum_i B_i\right\Vert_s
        \leq C_s\left\{\sum_j\left(\sum_i a_i b_{ij}\right)^s\right\}^{1/s}
        =C_s\mathcal Q_s(a,b).
    \end{align*}

    Substituting this estimate into \eqref{eq:DK-proxy-sums} proves \eqref{eq:abstract-main-bound}.

    Now suppose that $I$ is countably infinite.
    Choose increasing finite sets $I_r\uparrow I$, put
    \begin{align*}
        \mathcal G_r=\sigma(X_i:i\in I_r),
        \qquad H_r=\mathbb{E}[H\mid\mathcal G_r].
    \end{align*}

    Since $H$ is measurable with respect to $\sigma(X_i:i\in I)$, the martingale convergence theorem gives $H_r\to H$ in $L^p$.
    In particular, $\sigma_r^2:=\operatorname{Var} H_r\to\sigma^2>0$, so $\sigma_r>0$ for all sufficiently large $r$.

    The same conditional-expectation argument gives for $i,j\in I_r$,
    \begin{align}
        D_iH_r&=\mathbb{E}[D_iH\mid\mathcal G_r\vee\sigma(X_i')],\label{eq:conditional-first-r}\\
        D_iD_jH_r&=\mathbb{E}[D_iD_jH\mid\mathcal G_r\vee\sigma(X_i',X_j')], \qquad i\ne j.\label{eq:conditional-second-r}
    \end{align}

    Thus conditional Jensen inequality shows that the same $a_i,c_i$, and $b_{ij}$ dominate the resampling differences of $H_r$.  Define
    \begin{align*}
        \mathcal Q_{s,r}=\left(\sum_{j\in I_r}\left(\sum_{i\in I_r}a_i b_{ij}\right)^s\right)^{1/s}.
    \end{align*}

    All summands are nonnegative, and hence $\mathcal Q_{s,r}\leq \mathcal Q_s(a,b)$.  The result in finite-index case yields
    \begin{align}\label{eq:finite-approx-bound}
        d_{\mathrm K}\!\left(\frac{H_r-\mathbb{E} H_r}{\sigma_r},Z\right)
        \leq\frac{C_s\mathcal Q_s(a,b)}{\sigma_r^2}.
    \end{align}

    Let $W_r=(H_r-\mathbb{E} H_r)/\sigma_r$ and $W=(H-\mathbb{E} H)/\sigma$.  Then $W_r\to W$ in probability.  For every $z\in\mathbb{R}$ and $\delta>0$,
    \begin{align*}
        \mathbb{P}(W\leq z)
        \leq \mathbb{P}(W_r\leq z+\delta)+\mathbb{P}(|W_r-W|>\delta),
    \end{align*}
    and the analogous lower bound follows with $z-\delta$.  Since
    \begin{align*}
        \sup_z|\Phi(z+\delta)-\Phi(z)|\leq\frac{\delta}{\sqrt{2\pi}},
    \end{align*}
    we obtain
    \begin{align*}
        d_{\mathrm K}(W,Z)
        \leq d_{\mathrm K}(W_r,Z)+\mathbb{P}(|W_r-W|>\delta)+\frac{\delta}{\sqrt{2\pi}}.
    \end{align*}

    Take $r\to\infty$ in \eqref{eq:finite-approx-bound}, then let $\delta\downarrow0$. 
    This completes the proof for countable $I$.
\end{proof}

\subsection{Proofs for Theorem~\ref{thm:volume}--\ref{thm:euler}}\label{subsec:geometric-applications}

\begin{proof}[Proof of Theorem~\ref{thm:volume}]
    For $w\in L_n:=\{-n,\ldots,n-1\}^d$, let
    \begin{align*}
        \xi_w=\mu_{\operatorname{Vol}}(B_w)\leq 1.
    \end{align*}

    Since distinct cubes $B_w$, $w\in L_n$, overlap only on sets of Lebesgue measure zero,
    \begin{align*}
        \mu_{\operatorname{Vol}}(\Lambda_n)
        =\operatorname{Vol}\!\left(\Lambda_n\cap\{f\geq\ell\}_\infty\right)
        =\sum_{w\in L_n}\operatorname{Vol}\!\left(B_w\cap\{f\geq\ell\}_\infty\right)
        =\sum_{w\in L_n}\xi_w.
    \end{align*}

    In \eqref{eq:pstar-conjugacy}, take
    \begin{align}\label{eq:volume-pq}
        p=3s,\qquad p^*=3s/2.
    \end{align}

    Because $\beta>3sd$ is strict, choose $\eta>0$ so small that $\beta-\eta>3sd$.  
    By Lemma~\ref{lem:volume-envelope}, for every $v\in\mathbb Z^d$ and $w\in L_n$, there exists a nonnegative random variable $Y_v(w)$ satisfying
    \begin{align*}
        |D_v\xi_w(X)|\leq Y_v(w),\qquad
        \mathbb E\bigl(Y_v(w)\bigr)^r
        \leq C_r(1+|v-w|)^{-\beta+\eta},\qquad r\in\{p,p^*\}.
    \end{align*}

    For $u\ne v$, as $W\stackrel d=W^{(u)}$, the same result holds for $f^{(u)}=q*W^{(u)}$, which verifies the assumptions in Lemma~\ref{lem:local-sum}.
    Take $\theta_p=\theta_{p^*}=\beta-\eta$, then Lemma~\ref{lem:local-sum} implies \eqref{eq:local-general}.

    By \citet[Theorem~2.8(1)]{McAuleyQuasi2026}, under Assumption~\ref{ass:McAuley}, 
    $\operatorname{Var}\mu_{\operatorname{Vol}}(\Lambda_n)\asymp|\Lambda_n|=(2n)^d$.
    Since $|L_n|=(2n)^d$, \eqref{eq:local-rate} now gives \eqref{eq:vol-rate}.
\end{proof}

\begin{proof}[Proof of Theorem~\ref{thm:surface}]
    Recall that $\mu_{\operatorname{SA}}(B)=\mathcal H^{d-1}(B\cap\partial \{f\geq\ell\}_\infty)$. 
    For each fixed cubical face $F$ of dimension at most $d-1$,
    the inclusion $\partial \{f\geq\ell\}_\infty\subseteq\{f=\ell\}$ and Tonelli's theorem give
    \begin{align*}
        \mathbb E\mathcal H^{d-1}(F\cap\partial \{f\geq\ell\}_\infty)\leq\int_F\mathbb P(f(x)=\ell)\,d\mathcal H^{d-1}(x)=0.
    \end{align*}

    Thus all shared faces in $\Lambda_n$ contribute zero to $\mu_{\operatorname{SA}}(\Lambda_n)$ almost surely; see also \citet[Lemma~4.7]{McAuley2026a}. 
    Denote  $\xi_w=\mu_{\operatorname{SA}}(B_w)$ and $L_n=\{-n,\ldots,n-1\}^d$, we obtain
    \begin{align*}
        \mu_{\operatorname{SA}}(\Lambda_n)
        =\sum_{w\in L_n}\mu_{\operatorname{SA}}(B_w)
        =\sum_{w\in L_n}\xi_w.
    \end{align*}  
  
    Choose
    \begin{align}\label{eq:surface-pq}
        p=\frac{3sk'}{k'+s},\qquad p^*=\frac{3sk'}{2k'-s}.
    \end{align}

    Since $p<k'$, \citet[Proposition~4.6]{McAuley2026a} gives
    \begin{align*}
        \mathbb E|\xi_w|^{k'}
        \leq\mathbb E\bigl(\mathcal H^{d-1}(B_w\cap\{f=\ell\})\bigr)^{k'}
        \leq C,\qquad \mu_{\operatorname{SA}}(\Lambda_n)\in L^p.
    \end{align*}

    By Lemma~\ref{lem:surface-envelope}, for every $\eta>0$,
    $v\in\mathbb Z^d$, and $w\in L_n$, there exists a nonnegative
    random variable $Y_v(w)$ satisfying
    \begin{align*}
        |D_v\xi_w(X)|\leq Y_v(w),\qquad
        \mathbb E\bigl(Y_v(w)\bigr)^r
        \leq C_{r,\eta}(1+|v-w|)^{-\theta_r},\qquad r\in\{p,p^*\},
    \end{align*}
    where $\theta_r:=\beta(k'-r)/(3k'-1)-\eta$.
    By \eqref{eq:beta-SA}, choose $\eta>0$ sufficiently small that
    \begin{equation}\label{eq:surface-balance}
        \frac{\theta_p}{p}=\frac{\beta(k'-2s)}{3s(3k'-1)}-\frac{\eta}{p}>d,
        \qquad \frac{\theta_{p^*}}{p^*}=\frac{2\beta(k'-2s)}{3s(3k'-1)}-\frac{\eta}{p^*}>2d.
    \end{equation}

    For $u\ne v$, as $W\stackrel d=W^{(u)}$, the same result holds for $f^{(u)}=q*W^{(u)}$, which verifies the assumptions in Lemma~\ref{lem:local-sum}.
    Thus Lemma~\ref{lem:local-sum} implies \eqref{eq:local-general}. 
    Since \eqref{eq:beta-SA} implies \eqref{eq:SA-positive-decay}, Proposition~\ref{prop:SA-positive} gives
    \begin{align*}
        \operatorname{Var}\mu_{\operatorname{SA}}(\Lambda_n)
        \geq c|L_n|=c(2n)^d
    \end{align*}
    for all sufficiently large $n$. Hence \eqref{eq:local-general} yields
    \begin{align*}
        d_{\mathrm K}\!\left(
        \frac{\mu_{\operatorname{SA}}(\Lambda_n)-\mathbb E\mu_{\operatorname{SA}}(\Lambda_n)}
        {\sqrt{\operatorname{Var}\mu_{\operatorname{SA}}(\Lambda_n)}},Z\right)
        \leq\frac{C|L_n|^{1/s}}{\operatorname{Var}\mu_{\operatorname{SA}}(\Lambda_n)}
        \leq Cn^{-d(1-1/s)},
    \end{align*}
    proving \eqref{eq:SA-rate}.
\end{proof}

\begin{proof}[Proof of Theorem~\ref{thm:euler}]
    For fixed $n$, recall \eqref{eq:EC-cube-score} and define
    \begin{align*}
        \xi_w(X)=\mu_{\operatorname{EC}}(B_w;f)
        -\mu_{\operatorname{EC}}(\Gamma_{n,w};f),\qquad w\in I_n.
    \end{align*}

    Thus by Lemma~\ref{lem:EC-decomp}, $\mu_{\operatorname{EC}}(\Lambda_n)=\sum_{w\in I_n}\xi_w(X)$.

    Choose
    \begin{equation}\label{eq:euler-pq}
        p=\frac{3s(k-1)}{k+s-1},\qquad p^*=\frac{3s(k-1)}{2k-s-2}.
    \end{equation}

    Since $k-1>2s$, we have $2s<p<k-1$, $1<p^*\leq p$, and $1/p+1/p^*=1/s$. 
    By \eqref{eq:EC-score-moment},
    \begin{align*}
        \sup_{n\geq1}\sup_{w\in I_n}\mathbb E|\xi_w(X)|^{k-1}
        \leq C,\qquad \mu_{\operatorname{EC}}(\Lambda_n)\in L^p.
    \end{align*}

    By Lemma~\ref{lem:EC-envelope}, for every $\eta>0$, $v\in\mathbb Z^d$, and $w\in I_n$, 
    there exists a nonnegative random variable $Y_v(w)$ satisfying
    \begin{align*}
        |D_v\xi_w(X)|\leq Y_v(w),\qquad
        \mathbb E\bigl(Y_v(w)\bigr)^r
        \leq C_{r,\eta}(1+|v-w|)^{-\theta_r},\qquad r\in\{p,p^*\},
    \end{align*}
    where $\theta_r:=\beta(k-1-r)/(k-1)-\eta$
    and the constants are uniform in $n,v,w$. 
    By \eqref{eq:beta-EC}, choose $\eta>0$ sufficiently small that
    \begin{equation}\label{eq:euler-balance}
        \frac{\theta_p}{p}=\frac{\beta(k-2s-1)}{3s(k-1)}-\frac{\eta}{p}>d,
        \qquad \frac{\theta_{p^*}}{p^*}=\frac{2\beta(k-2s-1)}{3s(k-1)}-\frac{\eta}{p^*}>2d.
    \end{equation}

    For $u\ne v$, as $W\stackrel d=W^{(u)}$, the same result holds for $f^{(u)}=q*W^{(u)}$, which verifies the assumptions in Lemma~\ref{lem:local-sum}.
    Thus Lemma~\ref{lem:local-sum} implies \eqref{eq:local-general} with $L=I_n$.

    Since \eqref{eq:beta-EC} implies \eqref{eq:EC-positive-decay}, Proposition~\ref{prop:EC-positive} gives
    \begin{align*}
        \operatorname{Var}\mu_{\operatorname{EC}}(\Lambda_n)
        \geq c|I_n|=c(2n)^d
    \end{align*}
    for all sufficiently large $n$. Hence \eqref{eq:local-general} yields
    \begin{align*}
        d_{\mathrm K}\!\left(
        \frac{\mu_{\operatorname{EC}}(\Lambda_n)-\mathbb E\mu_{\operatorname{EC}}(\Lambda_n)}
        {\sqrt{\operatorname{Var}\mu_{\operatorname{EC}}(\Lambda_n)}},Z\right)
        \leq\frac{C|I_n|^{1/s}}{\operatorname{Var}\mu_{\operatorname{EC}}(\Lambda_n)}
        \leq Cn^{-d(1-1/s)},
    \end{align*}
    proving the upper bound in \eqref{eq:EC-rate}.

    To prove the lower bound in \eqref{eq:EC-rate}, denote $\sigma_n=\sqrt{\operatorname{Var}\mu_{\operatorname{EC}}(\Lambda_n)}$.
    Proposition~\ref{prop:EC-positive} gives $\sigma_n\asymp n^{d/2}$ for all sufficiently large $n$.
    Chebyshev's inequality yields
    \begin{align*}
        \mathbb P\!\left(\left|\mu_{\operatorname{EC}}(\Lambda_n)-\mathbb E\mu_{\operatorname{EC}}(\Lambda_n)\right|\leq2\sigma_n\right)
        \geq\frac34.
    \end{align*}

    Note that $\mu_{\operatorname{EC}}\in\mathbb{Z}$. 
    An interval of length $4\sigma_n$ contains at most $4\sigma_n+2$ integers, 
    thus there exists $x=x_n\in\mathbb{Z}$, such that $\mathbb{P}(\mu_{\operatorname{EC}}(\Lambda_n)=x)\geq3/[4(4\sigma_n+2)]$. 
    Let $\tilde{\Phi}$ denote the distribution function of $\mathcal{N}(\mathbb E\mu_{\operatorname{EC}}(\Lambda_n),\sigma_n^2)$, which is continuous.
    Then for any $1>\delta>0$,
    \begin{align*}
        d_K\left(\frac{\mu_{\operatorname{EC}}(\Lambda_n)-\mathbb E\mu_{\operatorname{EC}}(\Lambda_n)}{\sigma_n},Z\right)
        &=d_K(\mu_{\operatorname{EC}}(\Lambda_n),\mathcal{N}(\mathbb E\mu_{\operatorname{EC}}(\Lambda_n),\sigma_n^2))\\
        &\geq\max\left\{\left|\mathbb{P}(\mu_{\operatorname{EC}}(\Lambda_n)\leq x)-\tilde{\Phi}(x)\right|,\left|\mathbb{P}(\mu_{\operatorname{EC}}(\Lambda_n)<x)-\tilde{\Phi}(x-\delta)\right|\right\}\\
        &\geq\frac{1}{2}\left|\mathbb{P}(\mu_{\operatorname{EC}}(\Lambda_n)\leq x)-\tilde{\Phi}(x)-\mathbb{P}(\mu_{\operatorname{EC}}(\Lambda_n)<x)+\tilde{\Phi}(x-\delta)\right|\\
        &\geq\frac{1}{2}\mathbb{P}(\mu_{\operatorname{EC}}(\Lambda_n)=x)-\frac{\tilde{\Phi}(x)-\tilde{\Phi}(x-\delta)}{2}\\
        &\geq cn^{-d/2}-\frac{\tilde{\Phi}(x)-\tilde{\Phi}(x-\delta)}{2}.
    \end{align*}

    Let $\delta\rightarrow0$, this proves the lower bound in \eqref{eq:EC-rate}.
\end{proof}

\section{Proofs of the auxiliary lemmas}\label{sec:lemma-proofs}
In this section, we prove the lemmas stated in Section~\ref{sec:mixed-moment-tools}.

\subsection{Proofs of Lemma~\ref{lem:doob-identities}--\ref{lem:proxy-differences}}\label{subsec:proof-exchangeable}
\begin{proof}[Proof of Lemma~\ref{lem:doob-identities}]
    For any $i\in[m]$,condition on $J=i$, interchange $X_i$ and $X_i'$.
    This changes $(W,W',D,|D|)$ into $(W',W,-D,|D|)$ and preserves the joint distribution.
    Averaging over $J$ proves \eqref{eq:doob-joint-symmetry}.

    Since $X_i'$ is independent of $X$, recall $Y_i=\mathbb{E}[H\mid\mathcal{F}_i]-\mathbb{E}[H\mid\mathcal{F}_{i-1}]$ is a martingale difference, 
    and $h_i$ is a measurable version of $Y_i$,
    \begin{align*}
        \mathbb E[Y_i^{(i)}\mid X]=\mathbb{E}[h_i(X_1,\cdots,X_{i-1},X_i')\mid X]=0.
    \end{align*}

    Hence
    \begin{align*}
        \mathbb{E}[D\mid X]=\frac{1}{m\sigma}\sum_{i=1}^mY_i=\frac{H}{m\sigma},
    \end{align*}
    which proves \eqref{eq:regression-doob}.

    Additionally we have
    \begin{align*}
        D_iY_i&=Y_i-Y_i^{(i)}\\
        &=\mathbb{E}[H\mid\mathcal{F}_i]-\mathbb{E}[H\mid\mathcal{F}_{i-1}]-\mathbb{E}[H(X^{\{i\}})\mid\sigma(\mathcal{F}_{i-1},X_i')]+\mathbb{E}[H(X^{\{i\}})\mid\mathcal{F}_{i-1}]\\
        &=\mathbb{E}[D_iH\mid X_1,\ldots,X_i,X_i'].
    \end{align*}

    Thus conditional Jensen inequality proves \eqref{eq:DiYi-contract}. 
    If $j<i$, the resampling difference in coordinate $j$ commutes with integration over future coordinates, a similar argument giving
    \begin{align*}
        D_jD_iY_i=\mathbb{E}[D_jD_iH\mid X_1,\ldots,X_i,X_i',X_j'];
    \end{align*}
    conditional Jensen inequality gives the first part of \eqref{eq:DjDiYi-contract}.  
    If $j>i$, $Y_i$ does not depend on $X_j$, thus $D_jD_iY_i=0$.
\end{proof}

\begin{proof}[Proof of Lemma~\ref{lem:proxy-differences}]
    First take $j\ne i$.
    Put $U=D_iY_i$ and $T=D_iH$, and let $U^{(j)}$ denote $U$ with $X_j$ replaced by $X_j'$.
    Since $X_i'$ is independent of $\sigma(X,X_j')$,
    \begin{align*}
        D_jA_i=\mathbb E[D_j(UT)\mid X,X_j'].
    \end{align*}

    By product-difference identity
    \begin{align*}
        D_j(UT)=(D_jU)T+U^{(j)}D_jT,
    \end{align*}
    conditional Jensen inequality, mixed H\"older inequality and \eqref{eq:pstar-conjugacy},
    \begin{align*}
        \Vert D_jA_i\Vert_s
        \leq \Vert D_jU\Vert_{p^*}\Vert T\Vert_p+\Vert U^{(j)}\Vert_p\Vert D_jT\Vert_{p^*}
        \leq 2a_ib_{ij}.
    \end{align*}

    The same calculation applies to $B_i$ because $|D_j|U||\leq|D_jU|$.

    When $j=i$, $A_i$ is a function of $X$ alone, so $A_i(X^{\{i\}})$ has the same distribution as $A_i(X)$. The triangle inequality therefore gives
    \begin{align*}
        \Vert D_iA_i\Vert_s\leq 2\Vert A_i\Vert_s\leq 2\Vert D_iY_i\Vert_p\Vert D_iH\Vert_{p^*}\leq 2a_ic_i=2a_ib_{ii}.
    \end{align*}

    Similarly we have 
    \begin{align*}
        \Vert D_iB_i\Vert_s\leq 2a_ic_i=2a_ib_{ii}.
    \end{align*}
\end{proof}

\begin{proof}[Proof of Lemma~\ref{lem:uniform-low-moment}]
    This is a consequence of the martingale moment inequality of \citet[Theorem~2]{vonBahrEsseen1965}, 
    applied to the L\'evy martingale of $G$. 

    First suppose that $I=[m]$ be a finite index set. 
    Let $\mathcal F_0$ be the trivial sigma-field and 
    \begin{align*}
        \mathcal F_j=\sigma(X_1,\ldots,X_j),\qquad
        e_j=\mathbb E[G\mid\mathcal F_j]-\mathbb E[G\mid\mathcal F_{j-1}],
        \qquad S_j=\sum_{k=1}^j e_k,
    \end{align*}
    with $S_0=0$. Then $\mathbb E[e_j\mid\mathcal F_{j-1}]=0$ and $S_m=G-\mathbb EG$. Independence gives
    \begin{align*}
        \mathbb E[G(X^{\{j\}})\mid\mathcal F_j]=\mathbb E[G\mid\mathcal F_{j-1}],
        \qquad e_j=\mathbb E[D_jG\mid\mathcal F_j].
    \end{align*}

    Hence conditional Jensen's inequality yields
    \begin{equation}\label{eq:low-moment-increment}
        \mathbb E|e_j|^s\leq\mathbb E|D_jG|^s.
    \end{equation}

    Then von Bahr--Esseen inequality gives 
    \begin{align}\label{eq:low-moment-sum}
        \mathbb{E}[\left|G-\mathbb{E}G\right|^s]=\mathbb{E}[\left|S_m\right|^s]\leq 2\sum_{j=1}^m\mathbb{E}[\left|e_j\right|^s]\leq 2\sum_{j=1}^m\mathbb E|D_jG|^s,
    \end{align}
    which yields \eqref{eq:uniform-low-moment} when $I$ is finite.

    If $I$ is countably infinite, enumerate $(X)$ as $X_1,X_2,\ldots$ and use the same notations $\mathcal F_j,e_j,S_j$. 
    The conditional-expectation identity above and \eqref{eq:low-moment-increment} remain valid for every finite $m$. 
    Moreover, Doob's martingale convergence theorem gives
    \begin{align*}
        S_m=\mathbb E[G\mid\mathcal F_m]-\mathbb EG\longrightarrow G-\mathbb EG\qquad\text{in }L^s.
    \end{align*}

    As $G\in L^s$, the L.H.S of \eqref{eq:uniform-low-moment} is finite.
    Thus letting $m\to\infty$ in \eqref{eq:low-moment-sum} proves the claim with the same constant, even if the R.H.S of \eqref{eq:uniform-low-moment} is infinite.
\end{proof}

\subsection{Proof of Lemma~\ref{lem:local-sum}}\label{subsec:proof-spatial}

\begin{proof}
    By \eqref{eq:site-envelope}, for $u\in\{p,p^*\}$, $i\in\mathbb{Z}^d$,
    \begin{align*}
        \Vert D_i H\Vert_u\leq \sum_{w\in L}\Vert D_i\xi_w\Vert_u\leq C_u^{1/u}\sum_{w\in L}(1+\left|i-w\right|)^{-\theta_u/u}.
    \end{align*}

    Therefore we choose 
    \begin{align*}
        a_i=C_p^{1/p}\sum_{w\in L}(1+\left|i-w\right|)^{-\theta_p/p},\quad c_i=b_{ii}=C_{p^*}^{1/p^*}\sum_{w\in L}(1+\left|i-w\right|)^{-\theta_{p^*}/p^*}.
    \end{align*}

    Now we turn to $D_iD_jH$ for $i\not=j\in\mathbb{Z}^d$. Note that 
    \begin{align*}
        D_iD_j\xi_w=D_jD_i\xi_w=D_i\xi_w-D_i\xi_w(X^{\{j\}}).
    \end{align*}

    Then we have 
    \begin{align*}
        \Vert D_iD_j\xi_w\Vert_{p^*}\leq 2C_{p^*}^{1/p^*}\min\{(1+\left|i-w\right|)^{-\theta_{p^*}/p^*},(1+\left|j-w\right|)^{-\theta_{p^*}/p^*}\},
    \end{align*}
    and thus
    \begin{align*}
        \Vert D_iD_jH\Vert_{p^*}\leq 2C_{p^*}^{1/p^*}\sum_{w\in L}\min\{(1+\left|i-w\right|)^{-\theta_{p^*}/p^*},(1+\left|j-w\right|)^{-\theta_{p^*}/p^*}\}:=b_{ij}.
    \end{align*}

    By \eqref{eq:mixed-summability-criterion} and \eqref{eq:site-kernel-sums},
    \begin{align*}
        a_i\leq C_p^{1/p}K_p<\infty,\qquad
        c_i\leq C_{p^*}^{1/p^*}K_{p^*}<\infty.
    \end{align*}

    In the sequel, $C$ depends only on
    $(d,p,p^*,s,\theta_p,\theta_{p^*},C_p,C_{p^*})$
    and may change from line to line.

    Therefore, writing $S_j=\sum_{i\in\mathbb{Z}^d}a_ib_{ij}$, we have 
    \begin{align*}
        S_j&\leq C\left(c_j+\sum_{i\ne j}b_{ij}\right)\\
        &\leq C\sum_{w\in L}\sum_{i\ne j}
        \min\{(1+\left|i-w\right|)^{-\theta_{p^*}/p^*},
        (1+\left|j-w\right|)^{-\theta_{p^*}/p^*}\}\\
        &\qquad{}+C\sum_{w\in L}(1+\left|j-w\right|)^{-\theta_{p^*}/p^*}\\
        &\leq C\sum_{w\in L}(1+\left|j-w\right|)^{-\theta_{p^*}/(2p^*)}
        \sum_{i\ne j}(1+\left|i-w\right|)^{-\theta_{p^*}/(2p^*)}\\
        &\qquad{}+C\sum_{w\in L}(1+\left|j-w\right|)^{-\theta_{p^*}/p^*}\\
        &\leq C\sum_{w\in L}\left((1+\left|j-w\right|)^{-\theta_{p^*}/(2p^*)}
        +(1+\left|j-w\right|)^{-\theta_{p^*}/p^*}\right)\\
        &\leq C.
    \end{align*}

    By definition of $\mathcal{Q}_s(a,b)$ in \eqref{eq:Q-def-main},
    \begin{align}\label{eq:Q-order}
        \mathcal{Q}_s(a,b)
        =\left(\sum_j S_j^{s-1}S_j\right)^{1/s}
        \leq C
        \left(\sum_{j\in\mathbb{Z}^d}S_j\right)^{1/s},
    \end{align}
    where
    \begin{align}\label{eq:S-sum}
        \sum_{j\in\mathbb{Z}^d}S_j&\leq C\sum_{j\in\mathbb{Z}^d}\sum_{w\in L}
        \left((1+\left|j-w\right|)^{-\theta_{p^*}/(2p^*)}
        +(1+\left|j-w\right|)^{-\theta_{p^*}/p^*}\right)\notag\\
        &\leq C\left|L\right|.
    \end{align}

    Combining \eqref{eq:Q-order} and \eqref{eq:S-sum}, we prove \eqref{eq:Q-local}. 
    Therefore \eqref{eq:local-general} and \eqref{eq:local-rate} follow from Theorem~\ref{thm:abstract}.
\end{proof}

\subsection{Proofs of Lemma~\ref{lem:volume-envelope}--\ref{lem:EC-envelope}}\label{subsec:proof-geometric}

Consider a perturbation $p$ of a given Gaussian field $f$. 
Recall the finite-range unstable set defined in \citet{McAuley2026a} is 
\begin{equation}\label{eq:unstable-sets}
\begin{aligned}
    \mathcal{U}_{\mathrm{FR}}(f,p)= \mathcal{U}_{\mathrm{Loc}}(f,p)\cup
    \bigcup_{v \in \mathcal{U}_{\mathrm{Loc}}(f,p)}
    \left\{w\in \mathbb{Z}^d:B_w\overset{\{f\ge\ell\}_{<\infty}}{\longleftrightarrow}B_v
    \,\text{or}\,B_w\overset{\{f+p\ge\ell\}_{<\infty}}{\longleftrightarrow}B_v\right\},
\end{aligned}
\end{equation}
where
\begin{equation*}
\begin{split}
    \mathcal{U}_{\mathrm{Loc}}(f,p)=\bigl\{w\in\mathbb{Z}^d:\;&
    \exists\,t\in[0,1],\ S\in F(B_w),\ x\in S,\\
    &(f+tp)(x)=\ell,\ \nabla_S(f+tp)(x)=\mathbf{0}\bigr\}.
\end{split}
\end{equation*}

Here $\{f\ge\ell\}_{<\infty}$ denotes union of bounded components of excursion set $\{f\geq\ell\}$, 
$A\overset{B}{\longleftrightarrow}C$ means that a continuous curve in $B$ connects $A$ and $C$,
and $F(B_w)$ is the collection of strata of $B_w$. 

For any $w\in\mathbb{Z}^d\cap\mathcal{U}_{\mathrm{FR}}(f,p)^C$, the perturbation affects the target functionals (Vol, SA, EC) in $B_w$ only through geometric changes.
Thus it is important to control $\mathbb P(w\in\mathcal U_{\mathrm{FR}}(f,p))$ for each $w\in\mathbb Z^d$.
For details of these notations, see \citet[Section~3]{McAuley2026a}.

\begin{proof}[Proof of Lemma~\ref{lem:volume-envelope}]
    For the original resampling pair $(f,f+p_v)$, let
    \begin{align*}
        Y_v^{\operatorname{Vol}}(w)=\operatorname{Vol}\bigl(B_w\cap\{|f-\ell|\leq|p_v|\}\bigr)+\mathbf1_{\{w\in\mathcal U_{\mathrm{FR}}(f,p_v)\}}.
    \end{align*}

    The proof of \citet[Theorem~1.4]{McAuley2026a} gives 
    \begin{align*}
        |D_v\mu_{\operatorname{Vol}}(B_w)|\leq Y_v^{\operatorname{Vol}}(w)\leq 2,
        \quad \mathbb EY_v^{\operatorname{Vol}}(w)\leq C_\eta R(v,w)^{-\beta+\eta}.
    \end{align*}

    Hence, for every $r\geq1$,
    \begin{align*}
        \mathbb E\bigl(Y_v^{\operatorname{Vol}}(w)\bigr)^r
        \leq2^{r-1}\mathbb EY_v^{\operatorname{Vol}}(w)
        \leq C_{r,\eta}R(v,w)^{-\beta+\eta}.
    \end{align*}

    This proves \eqref{eq:vol-envelope}--\eqref{eq:vol-r-envelope}.
\end{proof}

\begin{proof}[Proof of Lemma~\ref{lem:surface-envelope}]
    Let $f^{(v)}=f+p_v$, $R=R(v,w)$, and 
    \begin{align*}
        \mathcal L(w)=\mathcal H^{d-1}(B_w\cap\{f=\ell\})+\mathcal H^{d-1}(B_w\cap\{f^{(v)}=\ell\}).
    \end{align*}
    
    Let $k'=\frac{k(k+1)}{2}-2$, then by \citet[Lemma~4.3 and Proposition~4.6]{McAuley2026a} (see also \citet{ArmentanoEtAl2023}),
    \begin{align}\label{SA-moment control}
        |D_v\mu_{\operatorname{SA}}(B_w)|\leq\mathcal L(w),\quad \mathbb E\mathcal L(w)^{k'}\leq C.
    \end{align}

    By \citet[Lemma~4.8]{McAuley2026a}, there exists a controlling random variable, which we denote by $\widehat Y_v(w)$, such that 
    \begin{align*}
        |D_v\mu_{\operatorname{SA}}(B_w)|\leq\widehat Y_v(w).
    \end{align*}
    
    Additionally, for any $\delta>0$, there exists $C=C_{\delta}$, such that 
    \begin{align*}
        \mathbb{E}[\widehat Y_v(w)]\leq C_{\delta}\mathbb{P}(w\in\mathcal{U}_{\mathrm{FR}}(f,p_v))^{(k'-1)/k'}+C_{\delta}M_2(w)^{\frac{\gamma-\delta}{2}},
    \end{align*}
    where $\gamma=\frac{k'-1}{3k'-1}$, and
    \begin{align*}
        M_2(w)=\sup_{|\alpha|,|\gamma|\leq2}\sup_{x,y\in w+\Lambda_2}
        \left|\partial_x^\alpha\partial_y^\gamma\operatorname{Cov}(p_v(x),p_v(y))\right|\leq CR^{-2\beta}
    \end{align*}
    by (M3) in Assumption~\ref{ass:McAuley},

    Denote
    \begin{align*}
        Y_v^{\operatorname{SA}}(w):=\min\{\widehat Y_v(w),\mathcal L(w)\}.
    \end{align*}
    
    Combining the estimates above, and by \citet[Lemma~3.5]{McAuley2026a}, we have 
    \begin{align}\label{eq:SA-envelope-endpoint-moment}
        |D_v\mu_{\operatorname{SA}}(B_w)|\leq Y_v^{\operatorname{SA}}(w),\quad \mathbb E\bigl(Y_v^{\operatorname{SA}}(w)\bigr)^{k'}\leq C,
    \end{align}
    and for $0<\delta<\gamma$ and $\eta_0>0$,
    \begin{align*}
        \mathbb EY_v^{\operatorname{SA}}(w)
        &\leq C_\delta\mathbb P\bigl(w\in\mathcal U_{\mathrm{FR}}(f,p_v)\bigr)^{(k'-1)/k'}
           +C_\delta M_2(w)^{(\gamma-\delta)/2}\\
        &\leq C_{\delta,\eta_0}
           \bigl(R^{-(\beta-\eta_0)(k'-1)/k'}
                  +R^{-\beta(\gamma-\delta)}\bigr).
    \end{align*}

    Since $(k'-1)/k'>\gamma$, choosing $\delta, \eta_0$ sufficiently small gives,
    for every $\varepsilon>0$,
    \begin{equation}\label{eq:SA-envelope-first-moment}
        \mathbb EY_v^{\operatorname{SA}}(w)
        \leq C_\varepsilon R^{-\beta(k'-1)/(3k'-1)+\varepsilon}.
    \end{equation}

    For $1\leq r<k'$, let $\theta=(k'-r)/(k'-1)$. 
    H\"older's inequality, together with \eqref{eq:SA-envelope-endpoint-moment}--\eqref{eq:SA-envelope-first-moment}, yields
    \begin{align}\label{eq:SA-envelope-interpolation}
        \mathbb E\bigl(Y_v^{\operatorname{SA}}(w)\bigr)^r
        &\leq\bigl(\mathbb EY_v^{\operatorname{SA}}(w)\bigr)^\theta\bigl(\mathbb E(Y_v^{\operatorname{SA}}(w))^{k'}\bigr)^{1-\theta}\notag\\
        &\leq C_{r,\varepsilon}R^{-\beta(k'-r)/(3k'-1)+\theta\varepsilon}.
    \end{align}

    Taking $\theta\varepsilon\leq\eta$ proves \eqref{eq:SA-r-envelope}.
\end{proof}

\begin{proof}[Proof of Lemma~\ref{lem:EC-decomp}]
    Enumerate $I_n=\{w_1,\ldots,w_M\}$ in increasing lexicographic order, where $M=(2n)^d$. 
    Then
    \begin{align*}
        B_{w_t}\cap\bigcup_{s>t}B_{w_s}=\bigcup_{\substack{j\in[d]\colon w_{t,j}<n-1}}\bigl(B_{w_t}\cap B_{w_t+e_j}\bigr)=\Gamma_{n,w_t}.
    \end{align*}

    Iterating \eqref{eq:EC-valuation} gives
    \begin{align*}
        \mu_{\operatorname{EC}}(\Lambda_n;f)
        &=\mu_{\operatorname{EC}}\!\left(\bigcup_{t=1}^M B_{w_t};f\right)\\
        &=\sum_{t=1}^M\left[\mu_{\operatorname{EC}}(B_{w_t};f)-\mu_{\operatorname{EC}}\!\left(B_{w_t}\cap\bigcup_{s>t}B_{w_s};f\right)\right]\\
        &=\sum_{t=1}^M\left[\mu_{\operatorname{EC}}(B_{w_t};f)-\mu_{\operatorname{EC}}(\Gamma_{n,w_t};f)\right]\\
        &=\sum_{w\in I_n}\xi_w(f).
    \end{align*}

    Each field obtained by resampling one or two blocks has the same distribution as $f$. 
    Taking first and second resampling differences proves \eqref{eq:EC-first-difference-sum}--\eqref{eq:EC-second-difference-sum}.
\end{proof}

\begin{proof}[Proof of Lemma~\ref{lem:EC-envelope}]
    Fix $v\in\mathbb Z^d$, $f^{(v)}=f+p_v$ be the new field obtained by exchange $X_v$ to $X_v'$.
    By \eqref{eq:EC-valuation},
    \begin{align*}
        D_v\xi_w(f)=\left[\mu_{\operatorname{EC}}(B_w,f)-\mu_{\operatorname{EC}}(B_w,f^{(v)})\right]-
        \left[\mu_{\operatorname{EC}}(\Gamma_{n,w},f)-\mu_{\operatorname{EC}}(\Gamma_{n,w},f^{(v)})\right].
    \end{align*}

    First assume that $w\notin\mathcal{U}_{\mathrm{FR}}(f,p_v)$, by \citet[Lemma~3.3]{McAuley2026a}, there exists a stratified homeomorphism $H_1$, such that 
    \begin{align*}
        H_1\left(\{f\geq\ell\}_{\infty}\cap B_w\right)=\{f^{(v)}\geq\ell\}_{\infty}\cap B_w,\quad H_1(\Gamma_{n,w})=\Gamma_{n,w}.
    \end{align*}

    Therefore we have for the union of upper strata $\Gamma_{n,w}$,
    \begin{align*}
        H_1\left(\{f\geq\ell\}_{\infty}\cap\Gamma_{n,w}\right)&=H_1\left(\{f\geq\ell\}_{\infty}\cap B_w\cap\Gamma_{n,w}\right)\\
        &=H_1\left(\{f\geq\ell\}_{\infty}\cap B_w\right)\cap H_1(\Gamma_{n,w})\\
        &=\{f^{(v)}\geq\ell\}_{\infty}\cap\Gamma_{n,w}.
    \end{align*}

    Combining these together yields
    \begin{align*}
        D_v\xi_w(f)=D_v\xi_w(f)\mathbf{1}_{w\in\mathcal{U}_{\mathrm{FR}}(f,p_v)}.
    \end{align*}

    In the sequel assume $w\in\mathcal{U}_{FR}(f,p_v)$. 
    For $u\in\{f,f^{(v)}\}$, let $N_w(u)=N_{\operatorname{Crit}}(B_w,u)$ count the critical points over all open strata of $B_w$, including its vertices. 
    For details of this notation, see \citet[Section~4]{McAuley2026a}. Additionally, by proof of \citet[Theorem~1.4]{McAuley2026a}, define
    \begin{equation}\label{eq:EC-cube-control}
        Y_v^{\operatorname{EC}}(w)=C_d\bigl(N_w(f)+N_w(f^{(v)})\bigr)\mathbf1_{\{w\in\mathcal U_{\mathrm{FR}}(f,p_v)\}}.
    \end{equation}

    By \citet[Lemma~4.12]{McAuley2026a},
    \begin{align}\label{eq:ec-upper bound}
        \max\left\{\left|\mu_{\operatorname{EC}}(B_w,u)\right|,\left|\mu_{\operatorname{EC}}(\Gamma_{n,w},u)\right|\right\}\leq \frac{C_d}{2}N_w(u).
    \end{align}

    Thus \eqref{eq:EC-local-envelope} follows from 
    \begin{align*}
        \left|D_v\xi_w(f)\right|&=\left|\xi_w(f)-\xi_w(f^{(v)})\right|\mathbf{1}_{w\in\mathcal{U}_{\mathrm{FR}}(f,p_v)}\\
        &\leq\left(\left|\xi_w(f)\right|+\left|\xi_w(f^{(v)})\right|\right)\mathbf{1}_{w\in\mathcal{U}_{\mathrm{FR}}(f,p_v)}\\
        &\leq Y_v^{\operatorname{EC}}(w)
    \end{align*}

    Since $f^{(v)}\stackrel d=f$, \citet[Theorem~1.2 and Remark~1.3]{GassStecconi2024} (see also \citet[Proposition~4.14]{McAuley2026a}), 
    applied to the finitely many positive-dimensional strata, and \citet[Lemma~3.5]{McAuley2026a}, applied to $(f,f+p_v)$, give
    \begin{align}\label{eq:EC-score-critical-bound}
        \mathbb E\bigl(N_w(f)+N_w(f^{(v)})\bigr)^{k-1}\leq C,
        \quad
        \mathbb P(w\in\mathcal U_{\mathrm{FR}}(f,p_v))\leq C_\delta(1+|v-w|)^{-\beta+\delta}.
    \end{align}

    For $1\leq r<k-1$, H\"older's inequality gives
    \begin{align*}
        \mathbb E\bigl(Y_v^{\operatorname{EC}}(w)\bigr)^r
        &\leq C_{d,r}\left\{\mathbb E\bigl(N_w(f)+N_w(f^{(v)})\bigr)^{k-1}\right\}^{r/(k-1)}\mathbb P(w\in\mathcal U_{\mathrm{FR}}(f,p_v))^{1-r/(k-1)}\\
        &\leq C_{r,\delta}(1+|v-w|)^{-(\beta-\delta)(k-r-1)/(k-1)}.
    \end{align*}

    Choosing $\delta>0$ so that $\delta(k-r-1)/(k-1)\leq\eta$ proves \eqref{eq:EC-r-envelope}. 
    Finally, \eqref{eq:EC-cube-score}, \eqref{eq:ec-upper bound} and \eqref{eq:EC-score-critical-bound} prove \eqref{eq:EC-score-moment}.
\end{proof}

\section{Variance of Surface Area and Euler characteristic}\label{sec:variance-positivity}

Under the respective assumptions of Propositions~\ref{prop:SA-positive}
and~\ref{prop:EC-positive}, \citet[Theorem~1.4]{McAuley2026a} gives
\begin{align*}
    \frac{\operatorname{Var}\mu_\star(\Lambda_n;f)}{(2n)^d}
    \longrightarrow \sigma_\star^2(\ell)\in[0,\infty),
    \qquad \star\in\{\operatorname{SA},\operatorname{EC}\}.
\end{align*}

It remains to prove that $\sigma_\star^2(\ell)$ are strictly positive.
The key observations in this section are Lemma~\ref{lem:Gaussian-response} and Lemma~\ref{lem:FR-nonnegative}.

\subsection{Positivity of $\sigma_\star^2(\ell)$}\label{subsec:Gaussian-response}

The following lemma is an extension of \citet[Lemma~5.2]{McAuley2026a},
which relates positivity of $\sigma_\star^2(\ell)$ to the expectation of first order difference of target functionals.

\begin{lemma}\label{lem:Gaussian-response}
    Let $\star\in\{\operatorname{SA},\operatorname{EC}\}$.
    Suppose that Assumption~\ref{ass:McAuley} holds, $q\geq0$, 
    together with \eqref{eq:SA-positive-decay} when $\star=\operatorname{SA}$ or \eqref{eq:EC-positive-decay} when $\star=\operatorname{EC}$.
    Let $0\ne h\in L^2(\mathbb R^d)$ is supported in $B=\Lambda_M$, where
    $M\in\mathbb N$ is fixed, and define the corresponding perturbation
    \begin{align*}
        \rho(x)=q*h(x)=\int_{\mathbb R^d}q(x-y)h(y)\,dy.
    \end{align*}
    
    Then the limit difference
    \begin{equation}\label{eq:fixed-direction-response}
        R_\star(h):=\lim_{n\to\infty}\left\{\mathbb E\mu_\star(\Lambda_n;f+\rho)-\mathbb E\mu_\star(\Lambda_n;f)\right\}
    \end{equation}
    exists and is finite.

    Further,
    \begin{align*}
        |R_\star(h)|^2\leq |B|e^{\|h\|_2^2}\sigma_\star^2(\ell).
    \end{align*}

    In particular, if $R_\star(h)\ne0$, then $\sigma_\star^2(\ell)>0$.
\end{lemma}

\begin{proof}
    Resample all white noise in $B$ independently, obtaining $f^B$, and define
    \begin{align*}
        \Delta_{B,n}=\mu_\star(\Lambda_n;f)-\mu_\star(\Lambda_n;f^B).
    \end{align*}

    Enumerate the unit-block indices in $B=\Lambda_M$ in increasing lexicographic order:
    \begin{align*}
        I_M=\{-M,\ldots,M-1\}^d=\{u_1,\ldots,u_N\}\qquad u_1\prec\cdots\prec u_N,\qquad N=(2M)^d.
    \end{align*}

    For $0\le j\le N$, let $f_j$ be the field obtained from $f$ by replacing $X_{u_i}$ with $X'_{u_i}$ for $1\le i\le j$, using the same independent copies as in the construction of $f^B$.
    Then $f_0=f$, $f_N=f^B$, and
    \begin{align*}
        \Delta_{B,n}&=\sum_{j=1}^N
        \left[\mu_\star(\Lambda_n;f_{j-1})-\mu_\star(\Lambda_n;f_j)\right]
        =:\sum_{j=1}^N\Delta_{j,n}.
    \end{align*}

    By independence of the white-noise blocks, for every $1\le j\le N$,
    \begin{align*}
        (f_{j-1},f_j)\overset{d}{=}(f,f^{(u_j)}).
    \end{align*}
    
    Thus conditions (2) and (3) of \citet[Theorem~2.2]{McAuley2026a}, verified in the proof of \citet[Theorem~1.4]{McAuley2026a}, imply that each $\Delta_{j,n}$ converges in $L^2$.
    As $N$ is fixed, we obtain
    \begin{equation}\label{eq:pure-block-L2}
        \Delta_{B,n}\longrightarrow\Delta_B\quad\text{in }L^2,\qquad
        \mathbb E\Delta_B=\lim_{n\to\infty}\mathbb E\Delta_{B,n}=0.
    \end{equation}

    The last equality follows from $f\overset{d}{=}f^B$.

    Define
    \begin{align*}
        \Gamma_h:=\exp\left\{W(h)-\frac12\|h\|_2^2\right\}.
    \end{align*}

    Then $\Gamma_h\in L^2$ and $\mathbb E\Gamma_h=1$.
    Recall that $W$ can be viewed as a centered Gaussian random element in some suitable Hilbert space, 
    with Cameron--Martin space canonically identified with $H=L^2(\mathbb R^d)$.
    Here we consider translation 
    \begin{align*}
        W\mapsto W+h,\quad f\mapsto f+\rho,
    \end{align*}
    where $(W+h)(\phi)=W(\phi)+\langle h,\phi\rangle_{H}, \phi\in H$.

    We use the block-measurable geometric versions in Remark~\ref{rem:field-measurability}.
    Let $\nu=\mathcal L(W)$ and $\nu_h=\mathcal L(W+h)$.
    The white-noise Cameron--Martin formula gives
    \begin{align*}
        \frac{d\nu_h}{d\nu}(W)=\exp\left\{W(h)-\frac12\|h\|_2^2\right\}=\Gamma_h.
    \end{align*}

    The Cameron--Martin formula gives
    \begin{align}\label{C--M formula}
        \mathbb E[\Gamma_h\Delta_{B,n}]
        &=\mathbb{E}[\Gamma_h\left(\mu_\star(\Lambda_n;f)-\mu_\star(\Lambda_n;f^B)\right)]\notag\\
        &=\mathbb E\bigl[\mu_\star(\Lambda_n;f+\rho)-\mu_\star(\Lambda_n;f^B)\bigr]\notag\\
        &=\mathbb E\mu_\star(\Lambda_n;f+\rho)-\mathbb E\mu_\star(\Lambda_n;f).
    \end{align}

    For the Cameron--Martin formula and its application to white noise, for example, see \citet[Chapter~4]{HairerSPDE2025}.
    
    The expectations are finite by Cauchy--Schwarz inequality, since $\mu_\star(\Lambda_n;f)$ and $\Gamma_h$ belong to $L^2$.
    By \eqref{eq:pure-block-L2} and \eqref{C--M formula},
    \begin{align*}
        \bigl|\mathbb E[\Gamma_h(\Delta_{B,n}-\Delta_B)]\bigr|
        \leq\|\Gamma_h\|_2\|\Delta_{B,n}-\Delta_B\|_2\longrightarrow0,
    \end{align*}
    thus
    \begin{align*}
        R_\star(h)=\mathbb E[\Gamma_h\Delta_B]<\infty.
    \end{align*}

    Denote $c=2M$ and $b=(M,\ldots,M)$. Recall $B=\Lambda_M$,
    \begin{align*}
        \widehat f(x)=f(cx-b),\qquad\mathcal G_B=\sigma\bigl(W|_{cv+B}:v\preceq0\bigr),
    \end{align*}
    then by the proof of \citet[Theorem~2.2]{McAuley2026a}
    \begin{equation}\label{eq:response-rescaled-variance}
        \lim_{n\to\infty}
        \frac{\operatorname{Var}\mu_\star(\Lambda_n;\widehat f)}{(2n)^d}
        =\mathbb E\bigl[\mathbb E(\widehat\Delta_{\mathbf{0}}\mid\mathcal G_B)^2\bigr],
    \end{equation}
    where $\widehat\Delta_{\mathbf{0}}$ is the limiting resampling difference for $\widehat f$ on $B_{\mathbf{0}}$.
    
    Similar to the proof of \citet[Lemma~5.2]{McAuley2026a}, for the scaling field,
    \begin{equation}\label{eq:response-functional-scaling}
        \begin{aligned}
            \mu_{\operatorname{SA}}(\Lambda_n;\widehat f)
            &=c^{1-d}\mu_{\operatorname{SA}}(c\Lambda_n-b;f),\\
            \mu_{\operatorname{EC}}(\Lambda_n;\widehat f)
            &=\mu_{\operatorname{EC}}(c\Lambda_n-b;f).
        \end{aligned}
    \end{equation}

    For the resampled field $f^B$, the same scaling holds. 
    Applying \eqref{eq:response-functional-scaling} to the resampling differences yields
    \begin{equation}\label{eq:response-difference-scaling}
        \widehat\Delta_{\mathbf{0}}=
        \begin{cases}
            c^{1-d}\Delta_B,&\star=\operatorname{SA},\\
            \Delta_B,&\star=\operatorname{EC}.
        \end{cases}
    \end{equation}

    Combining \eqref{eq:response-rescaled-variance}, \eqref{eq:response-functional-scaling} and \eqref{eq:response-difference-scaling},
    stationarity of $\mu_\star$ then gives
    \begin{equation}\label{eq:response-block-variance}
        \mathbb E\bigl[\mathbb E(\Delta_B\mid\mathcal G_B)^2\bigr]
        =\lim_{n\to\infty}
          \frac{\operatorname{Var}
            \mu_\star(c\Lambda_n-b;f)}{(2n)^d}
            =\lim_{n\to\infty}
          \frac{\operatorname{Var}
            \mu_\star(\Lambda_{cn};f)}{(2n)^d}
        =c^d\sigma_\star^2(\ell).
    \end{equation}

    As $h$ is supported in $B$, $\Gamma_h$ is
    $\mathcal G_B$-measurable. By \eqref{eq:response-block-variance}
    and Cauchy--Schwarz,
    \begin{align}\label{eq:response-conditional-CS}
        |R_\star(h)|^2
        &=\left|\mathbb E\bigl[\Gamma_h\mathbb E(\Delta_B\mid\mathcal G_B)\bigr]\right|^2\notag\\
        &\leq\|\Gamma_h\|_2^2\mathbb E\bigl[\mathbb E(\Delta_B\mid\mathcal G_B)^2\bigr]\notag\\
        &=|B|e^{\|h\|_2^2}\sigma_\star^2(\ell).
    \end{align}

    The coefficient is finite, so $R_\star(h)\ne0$ implies $\sigma_\star^2(\ell)>0$.
\end{proof}

\subsection{Finite-range unstable set}\label{subsec:FR-reduction}

For a deterministic field $g$, the notation $\{g\geq\ell\}_\infty$ means the union of all its unbounded excursion components, 
while 
\begin{align*}
    \{g\geq \ell\}_{<\infty}=\{g\geq \ell\}\textbackslash\{g\geq\ell\}_\infty.
\end{align*}

\begin{lemma}\label{lem:cubical-excursion-components}
    Let $g\in C^1(\mathbb R^d)$ and $\ell\in\mathbb R$.
    If $\ell$ is a regular value of $g$ on every cubical
    stratum of $B_{\mathbf{0}}=[0,1]^d$, then
    $\{g\geq\ell\}\cap B_{\mathbf{0}}$ has finitely many connected
    components.
\end{lemma}

\begin{proof}
    Denote $K=\{g\geq\ell\}\cap B_{\mathbf{0}}$.
    We show that every point of $K$ has a connected relatively open neighbourhood in $K$.

    Fix $x\in K$. If $g(x)>\ell$, choose a sufficiently small open ball $O_x$ on which $g>\ell$.
    Then $K\cap O_x=B_{\mathbf{0}}\cap O_x$ is convex.

    Suppose that $g(x)=\ell$. 
    Regularity on the stratum containing $x$ gives an index $j\in[d]$ such that
    \begin{align*}
        0<x_j<1,\qquad \partial_jg(x)\ne0.
    \end{align*}

    By the inverse function theorem, the map
    \[
        \Phi(y)
        =
        \bigl(y_1,\ldots,y_{j-1},
              g(y)-\ell,y_{j+1},\ldots,y_d\bigr)
    \]
    is a $C^1$ diffeomorphism on a neighbourhood $U$ of $x$.
    Shrinking $U$ if necessary, we may assume that $0<y_j<1$ throughout $U$. 
    Choose an open box $V$ containing $\Phi(x)$ with $V\subset\Phi(U)$, 
    and set $O_x=(\Phi|_U)^{-1}(V)$. 

    For $z\in V$, let $y=(\Phi|_U)^{-1}(z)\in O_x$.
    Then $z_i=y_i$ for $i\ne j$ and $z_j=g(y)-\ell$.
    Since $0<y_j<1$ throughout $U$,
    \begin{align*}
        y\in K\quad\Longleftrightarrow\quad z_i\in[0,1]\ \text{for all }i\ne j,\quad z_j\ge0.
    \end{align*}

    Then
    \begin{align*}
        \Phi(K\cap O_x)=V\cap\bigl([0,1]^{j-1}\times[0,\infty)\times[0,1]^{d-j}\bigr),
    \end{align*}
    which is convex. Hence $K\cap O_x$ is connected.

    As $g$ is continuous, $K$ is compact, 
    thus finitely many of these relatively open connected sets cover $K$. 
    Each of these sets lies in a single connected component of $K$, 
    so $K$ has only finitely many connected components.
\end{proof}

The next lemma shows that, for a nonnegative perturbation $p$,
we can simplify the expression of the finite-range unstable set $\mathcal U_{\mathrm{FR}}(g,p)$.
Recall its definition in \citet[Section~3]{McAuley2026a}, also see Subsection~\ref{subsec:proof-geometric}.

\begin{lemma}\label{lem:FR-nonnegative}
    Let $g,p\in C^2(\mathbb R^d)$ with $p\geq0$, and fix $\ell\in\mathbb R$. 
    Then
    \begin{equation}\label{u_FR identity}
        \mathcal U_{\mathrm{FR}}(g,p)
        =\mathcal U_{\mathrm{Loc}}(g,p)
        \cup\bigcup_{w\in\mathcal U_{\mathrm{Loc}}(g,p)}
        \left\{v\in\mathbb Z^d:B_v\overset{\{g\geq\ell\}_{<\infty}}{\longleftrightarrow}B_w\right\}.
    \end{equation}
\end{lemma}

\begin{proof}
    Let $A$ denote the right-hand side of \eqref{u_FR identity}.
    By definition, $A\subseteq\mathcal U_{\mathrm{FR}}(g,p)$.
    To prove the reverse inclusion, fix $v\notin A$.

    We first construct a finite region on which the perturbation is locally stable. 
    Define
    \begin{align*}
        D^+:=B_v\cup\bigcup_{u\in\mathbb Z^d:B_v\overset{\{g\geq\ell\}_{<\infty}}{\longleftrightarrow}B_u}B_u.
    \end{align*}

    Since $v\notin\mathcal U_{\mathrm{Loc}}(g,p)$,
    $\ell$ is a regular value of $g$ on every cubical
    stratum of $B_v$.
    Applying Lemma~\ref{lem:cubical-excursion-components}
    to $g(v+\cdot)$ shows that
    $\{g\geq\ell\}\cap B_v$ has finitely many connected
    components.

    Every full-space component of $\{g\geq\ell\}$ intersecting $B_v$ contains a component of $\{g\geq\ell\}\cap B_v$.
    Thus only finitely many bounded full-space components meet $B_v$,
    each is compact and intersects only finitely many unit cubes, 
    thus $D^+$ is a finite union of cubes. 
    Every cube in this union is locally stable; otherwise, $v$ would belong to the second part of $A$.

    By \citet[Lemma~3.1]{McAuley2026a}, there exists a stratified isotopy $H_t$ of $D^+$, 
    with $H_0$ equal to the identity, such that
    \begin{align*}
        H_t\bigl(\{g\geq\ell\}\cap D^+\bigr)=\{g+tp\geq\ell\}\cap D^+,\qquad 0\leq t\leq1.
    \end{align*}

    Since the isotopy preserves every cubical stratum,
    \begin{align*}
        H_t(B_v)=B_v,\qquad H_t(\partial D^+)=\partial D^+.
    \end{align*}

    We want to show that $v\notin\mathcal U_{\mathrm{FR}}(g,p)$.
    Suppose, for a contradiction, that there exist $w\in\mathcal U_{\mathrm{Loc}}(g,p)$ and a continuous curve
    \begin{align*}
        \gamma:[0,1]\longrightarrow\{g+p\geq\ell\}_{<\infty},\qquad\gamma(0)\in B_v,\quad \gamma(1)\in B_w.
    \end{align*}
    
    Denote $C$ be the bounded component of $\{g+p\geq\ell\}$ containing $\gamma$,
    \begin{align*}
        y_0=\gamma(0)\in C\cap B_v,\quad x_0=H_1^{-1}(y_0)\in\{g\geq\ell\}\cap B_v.
    \end{align*}

    By the isotopy identity and the nonnegativity of $p$,
    \begin{align*}
        H_t(x_0)\in\{g+tp\geq\ell\}\subseteq\{g+p\geq\ell\},
        \qquad 0\leq t\leq1.
    \end{align*}

    Since $H_0(x_0)=x_0$ and $H_1(x_0)=y_0\in C$, this path shows that $x_0\in C$. 
    Finally, $\{g\geq\ell\}\subseteq\{g+p\geq\ell\}$,
    so the original excursion component of $\{g\geq\ell\}$ containing $x_0$ is bounded, which we denote as $C_0$.

    Next we show that $x_0\notin\partial D^+$. 
    If not, there exists another cube $B_{u_0}$ containing $x_0$, but not in $D^+$.
    But we can construct a closed curve  containing $x_0$ in $\{g\geq\ell\}$ connecting $B_v$ and $B_{u_0}$,
    which leads to contradiction.
    As $H_1$ preserves $\partial D^+$,
    $y_0\in\operatorname{int}D^+$ as well.

    By definition of $D^+$ and $B_w$ is unstable,
    $B_w\cap\operatorname{int}D^+=\varnothing$.
    Hence $\gamma$ reaches $\partial D^+$.
    Let
    \begin{align*}
        t_0:=\min\{t\in[0,1]:\gamma(t)\in\partial D^+\}>0.
    \end{align*}

    Then $\gamma([0,t_0])\subseteq D^+, \gamma(t_0)\in\partial D^+$.
    As the isotopy $H_t$ preserves every strata, the pullback of $\gamma$
    \begin{align*}
        \eta(s):=H_1^{-1}(\gamma(s)),\qquad 0\leq s\leq t_0,
    \end{align*}
    lies in $\{g\geq\ell\}\cap D^+$, starts at $x_0$, and ends on $\partial D^+$.
    It therefore lies in the bounded component $C_0$.

    For every closed unit cube $B_u$ containing $\eta(t_0)$,
    the path $\eta$ gives
    \begin{align*}
        B_v\overset{\{g\geq\ell\}_{<\infty}}{\longleftrightarrow}B_u.
    \end{align*}

    By definition, all such cubes belong to $D^+$.
    Their union contains a neighbourhood of $\eta(t_0)$, contradicting $\eta(t_0)\in\partial D^+$.

    Thus $B_v$ cannot be connected to a locally unstable cube through a bounded component of $\{g+p\geq\ell\}$. 
    Together with $v\notin A$,
    this gives $v\notin\mathcal U_{\mathrm{FR}}(g,p)$ and proves the reverse inclusion.
\end{proof}

Lemma~\ref{lem:FR-nonnegative} removes the third term
\begin{align*}
    \bigcup_{v \in \mathcal{U}_{\mathrm{Loc}}(f,p)}
    \left\{w\in \mathbb{Z}^d:B_w\overset{\{f+p\ge\ell\}_{<\infty}}{\longleftrightarrow}B_v\right\},
\end{align*} 
in the expression of $\mathcal U_{\mathrm{FR}}(g,p)$.
Estimating the probability that $w$ lies in this set may require the additional condition $\ell<-\ell_c$; 
see \citet[Proposition~3.4 and Lemma~3.5]{McAuley2026a}. 

\subsection{Local estimates under chosen perturbation}\label{subsec:FR-perturbations}

Throughout this subsection, $C$ denotes a finite positive constant whose value may not be the same.

We now construct a nonnegative perturbation.
Assume $q\geq0$ and let
\begin{align*}
    a_q=\int_{\mathbb R^d}q(y)\,dy\in(0,\infty).
\end{align*}

Choose smooth cutoff functions $\xi_m\in C_c^\infty(\mathbb R^d)$,
$m\in\mathbb N$, such that
\begin{align}\label{property of xi}
    0\le\xi_m\le1,\qquad \xi_m=1\text{ on }\Lambda_m,\qquad \operatorname{supp}\xi_m\subseteq\Lambda_{m+1},\qquad \sup_{m\ge1}\|\partial^\alpha\xi_m\|_\infty<\infty
\end{align}
for every multi-index $\alpha$.
For a fixed $u>0$, specify the perturbation as
\begin{equation}\label{eq:smooth-white-noise-shift}
    h_m=\frac{u}{a_q}\xi_m,\qquad
    \rho_m(x)=\frac{u}{a_q}
       \int_{\mathbb R^d}q(y)\xi_m(x-y)\,dy.
\end{equation}

Then by \eqref{property of xi} we have
\begin{align}\label{eq:smooth-shift-regularity}
    0\leq\rho_m\leq u,\qquad\partial^\alpha\rho_m(x)
    &=\frac{u}{a_q}\int_{\mathbb R^d}q(y)\partial^\alpha\xi_m(x-y)\,dy,\notag\\
    \|\partial^\alpha\rho_m\|_\infty
    &\leq \frac{u}{a_q}\|q\|_1\|\partial^\alpha\xi_m\|_\infty\leq C_\alpha u.
\end{align}

In particular, $\sup_m\|\rho_m\|_{C^k(\mathbb R^d)}\leq C_k u$.

If $x\in\Lambda_{m-R}$ and $y\in\Lambda_R$, then $\xi_m(x-y)=1$. 
Hence for $x\in\Lambda_{m-R}$,
\begin{align}\label{eq:smooth-shift-core}
    0\leq u-\rho_m(x)\leq\frac{u}{a_q}\int_{\mathbb R^d\setminus\Lambda_R}q(y)\,dy,
    \quad |\partial^\alpha\rho_m(x)|\leq C_\alpha u\int_{\mathbb R^d\setminus\Lambda_R}q(y)\,dy,\quad |\alpha|\geq1.
\end{align}

For $v\notin\Lambda_{m+1}$, define
\begin{align*}
    L_{m,v}=\operatorname{dist}_\infty(B_v,\Lambda_{m+1}):=\inf_{x\in B_v,y\in\Lambda_{m+1}}\Vert x-y\Vert_{\infty}.
\end{align*}

When $L_{m,v}$ is sufficiently large, differentiating the kernel and using (M3) in Assumption~\ref{ass:McAuley} yield
\begin{align}\label{eq:SA-outside-perturbation}
    \|\rho_m\|_{C^2(v+\Lambda_2)}
    &:=
    \max_{|\alpha|\le2}
    \sup_{x\in v+\Lambda_2}|\partial^\alpha\rho_m(x)|
    \notag\\
    &\le
    \frac{u}{a_q}
    \sum_{|\alpha|\le2}
    \sup_{x\in v+\Lambda_2}
    \int_{\Lambda_{m+1}}
        |\partial^\alpha q(x-z)|\,dz \notag\\
    &\le
    \frac{u}{a_q}
    \sum_{|\alpha|\le2}
    \int_{|y|_\infty\ge L_{m,v}-2}
        |\partial^\alpha q(y)|\,dy \notag\\
    &\le Cu L_{m,v}^{-(\beta-d)},
\end{align}

For the Euler characteristic functional, recall that in \eqref{eq:EC-cube-score} we denote
\begin{equation}\label{eq:EC-window-score}
    \xi_{n,w}(g)=\mu_{\operatorname{EC}}(B_w;g)-\mu_{\operatorname{EC}}(\Gamma_{n,w};g),\qquad w\in I_n.
\end{equation}

Also recall that $N_w(g)=N_{\operatorname{Crit}}(B_w,g)$ counts stratified critical points in $B_w$,
see \citet[Section~4]{McAuley2026a} or the proof of Lemma~\ref{lem:EC-envelope}.

We next record the moment estimates needed below.

Under Assumption~\ref{ass:McAuley}, assume $q\geq0$, fix $u>0$, and let $k'=k(k+1)/2-2$. 
\eqref{eq:smooth-shift-regularity} and \citet[Proposition~4.6]{McAuley2026a} (see also \citet{ArmentanoEtAl2023}) give
\begin{equation}\label{eq:SA-shift-moments}
    \sup_{m\geq1}\sup_{w\in\mathbb Z^d}\mathbb E\left[\mathcal H^{d-1}(B_w\cap\{f+\rho_m=\ell\})^{k'}\right]\leq C_u.
\end{equation}

Moreover, under Assumption~\ref{ass:McAuley}, \citet[Lemma~4.12, Proposition~4.14]{McAuley2026a} (see also \citet{GassStecconi2024}),
together with the argument in the proof of Lemma~\ref{lem:EC-envelope}, yield, for the field $f+\rho_m$,
\begin{align}\label{eq:EC-shift-moments}
    \sup_{m\geq1}\sup_{w\in\mathbb Z^d}\mathbb E N_w(f+\rho_m)^{k-1}\leq C_u,\quad
    \sup_{m,n\geq1}\sup_{w\in I_n}\mathbb E|\xi_{n,w}(f+\rho_m)|^{k-1}\leq C_u,
\end{align}
here we apply the \eqref{eq:smooth-shift-regularity} to ensure $C_u$ is uniform in $m$.

The same moment bounds hold with $f+\rho_m$ replaced by $f$ or $f+u$.

\begin{lemma}\label{lem:shift-local-estimates}
    Under Assumption~\ref{ass:McAuley}, assume $q\geq0$ and fix $u>0$.
    For every sufficiently small $\eta>0$, there is a finite constant $C_{u,\eta}$, independent of $m,n,w$, 
    such that, for all sufficiently large $L_{m,w}$,
    \begin{equation}\label{eq:FR-exterior-tail}
        \mathbb P(w\in\mathcal U_{\mathrm{FR}}(f,\rho_m))\leq C_{u,\eta}L_{m,w}^{-\beta+2d+\eta}.
    \end{equation}

    If \eqref{eq:SA-positive-decay} holds, there exists $\gamma_{\operatorname{SA}}>d$ and a finite constant $C_u$, independent of $m,w$, 
    such that
    \begin{equation}\label{eq:SA-summable-exterior}
        \mathbb E\left|\mu_{\operatorname{SA}}(B_w;f+\rho_m)-\mu_{\operatorname{SA}}(B_w;f)\right|
        \leq C_u L_{m,w}^{-\gamma_{\operatorname{SA}}}.
    \end{equation}

    If \eqref{eq:EC-positive-decay} holds, there exists $\gamma_{\operatorname{EC}}>d$ and a finite constant $C_u$, independent of $m,n,w$, 
    such that, uniformly over $n$ with $w\in I_n$,
    \begin{equation}\label{eq:EC-summable-exterior}
        \mathbb E\left|\xi_{n,w}(f+\rho_m)-\xi_{n,w}(f)\right|
        \leq C_u L_{m,w}^{-\gamma_{\operatorname{EC}}}.
    \end{equation}
\end{lemma}

\begin{proof}
    To prove \eqref{eq:FR-exterior-tail}, we follow the argument for \citet[Lemma~5.3]{McAuley2026a}.
    Here \eqref{eq:SA-outside-perturbation} replaces \citet[equation~(5.15)]{McAuley2026a}, 
    and Lemma~\ref{lem:FR-nonnegative} allows us to omit the third term in \citet[equation~(5.16)]{McAuley2026a}, 
    which requires $\ell<-\ell_c$.
    
    Writing $L=L_{m,w}$, we obtain
    \begin{align*}
        \mathbb P\bigl(w\in\mathcal U_{\mathrm{FR}}(f,\rho_m)\bigr)
        \le C_{u,\epsilon}L^{-\beta+2d+\epsilon(\beta-d)}+C_NL^{-N}
    \end{align*}
    for $0<\epsilon<1$ and any $N>0$, uniformly in $m,w$ for
    sufficiently large $L$.
    Choosing $\epsilon$ sufficiently small and $N$ sufficiently
    large proves \eqref{eq:FR-exterior-tail}.

    For surface area, apply \citet[Lemma~4.8]{McAuley2026a},
    where the covariance term vanishes since $\rho_m$ is deterministic.
    For sufficiently small $\delta>0$,
    \begin{align}\label{eq:SA-shift-comparison}
        \mathbb E\left|\mu_{\operatorname{SA}}(B_w;f+\rho_m)-\mu_{\operatorname{SA}}(B_w;f)\right|
        &\leq C_{u,\delta}\left[\mathbb P(w\in\mathcal U_{\mathrm{FR}}(f,\rho_m))^{(k'-1)/k'}+\|\rho_m\|_{C^1(B_w)}^{(k'-1)/(3k'-1)-\delta}\right]\notag\\
        &\leq C_{u,\eta,\delta}\left[L^{-(\beta-2d-\eta)(k'-1)/k'}+L^{-(\beta-d)((k'-1)/(3k'-1)-\delta)}\right],
    \end{align}
    where the last inequality follows from \eqref{eq:SA-outside-perturbation} and \eqref{eq:FR-exterior-tail}.
    By \eqref{eq:SA-positive-decay}, both decay exponents exceed
    $d$ when $\eta,\delta>0$ are sufficiently small.
    Taking $\gamma_{\operatorname{SA}}$ to be their minimum proves
    \eqref{eq:SA-summable-exterior}.

    For Euler characteristic, which is a topological invariant,
    by \citet[Lemma~3.3]{McAuley2026a},
    \begin{align*}
        \xi_{n,w}(f+\rho_m)=\xi_{n,w}(f)\quad\text{on }\{w\in\mathcal U_{\mathrm{FR}}(f,\rho_m)\}^c
    \end{align*}

    By H\"older's inequality, \eqref{eq:EC-shift-moments} and \eqref{eq:FR-exterior-tail},
    \begin{align*}
        &\mathbb E\left|\xi_{n,w}(f+\rho_m)-\xi_{n,w}(f)\right|\\
        \leq&\left(\|\xi_{n,w}(f+\rho_m)\|_{k-1}+\|\xi_{n,w}(f)\|_{k-1}\right)\mathbb P(\{w\in\mathcal U_{\mathrm{FR}}(f,\rho_m)\})^{(k-2)/(k-1)}\\
        \leq&C_{u,\eta}L^{-(\beta-2d-\eta)(k-2)/(k-1)},
    \end{align*}
    uniformly in $m,n,w$ with $w\in I_n$.
    By \eqref{eq:EC-positive-decay}, we may choose $\eta>0$ sufficiently small that
    \begin{align*}
        \gamma_{\operatorname{EC}}:=(\beta-2d-\eta)\frac{k-2}{k-1}>d.
    \end{align*}

    This proves \eqref{eq:EC-summable-exterior}.
\end{proof}

\subsection{Surface area variance}\label{subsec:SA-variance-positive}

\begin{proof}[Proof of Proposition~\ref{prop:SA-positive}]
    By the proof of \citet[Theorem~1.3]{McAuley2026a},
    \begin{align*}
        c_{\operatorname{SA}}(\ell)=\mathbb E\mu_{\operatorname{SA}}(B_{\mathbf{0}};f)>0.
    \end{align*}

    Therefore, by stationarity and the additivity of $\mu_{\operatorname{SA}}$ in \citet[Lemma~4.7]{McAuley2026a}, 
    for every $n\in\mathbb N$,
    \begin{align}\label{eq:SA-unshifted-density}
        \mathbb E\mu_{\operatorname{SA}}(\Lambda_n;f)
        =\sum_{w\in I_n}\mathbb E\mu_{\operatorname{SA}}(B_w;f)
        =(2n)^d c_{\operatorname{SA}}(\ell).
    \end{align}

    Since $\partial\{f+\rho_m\geq\ell\}_{\infty}\subset\{f+\rho_m=\ell\}$, 
    the Kac--Rice calculation in the proof of \citet[Proposition~4.6]{McAuley2026a} gives,
    for bounded Borel set $E$,
    \begin{align}\label{eq:SA-shift-Kac-Rice}
        \mathbb E\mu_{\operatorname{SA}}(E;f+\rho_m)
        &\leq\mathbb E\mathcal H^{d-1}(E\cap\{f+\rho_m=\ell\})\notag\\
        &\leq\int_E\varphi(\ell-\rho_m(x))\mathbb E\left[\left|\nabla f(\mathbf{0})\right|+\left|\nabla\rho_m(x)\right|\right]\,dx,
    \end{align}
    where $\varphi$ is the density of Gaussian random variable $f(x)$.

    By \eqref{eq:smooth-shift-regularity},
    $\|\nabla\rho_m\|_\infty\le Cu$, where $C$ is independent of $u$ and $m$.
    Choose a sufficiently large fixed $u>0$ such that $u/2>\ell$ and
    \begin{align}\label{eq:varphi upper bound}
        \varphi(\ell-u/2)\bigl(\mathbb E|\nabla f(\mathbf{0})|+Cu\bigr)\leq \frac{c_{\operatorname{SA}}(\ell)}{4}.
    \end{align}

    Next choose a fixed integer $R\ge2$ such that \eqref{eq:SA-summable-exterior} applies whenever $L_{m,w}\ge R$, and
    \begin{align}\label{eq:R-inequality}
        \int_{\Lambda_R}q(y)\,dy\ge\frac{a_q}{2}.
    \end{align}

    Now fix $m>2R$, let $r=m-R, s=m+1+R$,
    and assume $n\geq s$. 
    By additivity of $\mu_{\operatorname{SA}}$ over unit cubes,
    \begin{align}\label{eq:SA-three-regions}
        \mathbb E\mu_{\operatorname{SA}}(\Lambda_n;f+\rho_m)-\mathbb E\mu_{\operatorname{SA}}(\Lambda_n;f)
        = I_1+I_2+I_3,
    \end{align}
    where
    \begin{align*}
        I_1&:=\sum_{w\in I_r}\mathbb E\!\left[\mu_{\operatorname{SA}}(B_w;f+\rho_m)-\mu_{\operatorname{SA}}(B_w;f)\right],\\
        I_2&:=\sum_{w\in I_s\setminus I_r}\mathbb E\!\left[\mu_{\operatorname{SA}}(B_w;f+\rho_m)-\mu_{\operatorname{SA}}(B_w;f)\right],\\
        I_3&:=\sum_{w\in I_n\setminus I_s}\mathbb E\!\left[\mu_{\operatorname{SA}}(B_w;f+\rho_m)-\mu_{\operatorname{SA}}(B_w;f)\right].
    \end{align*}

    For the interior term $I_1$, \eqref{eq:smooth-shift-core} and \eqref{eq:R-inequality} give $\rho_m\ge u/2$ on $\Lambda_r$. 
    Hence
    \begin{align*}
        \varphi(\ell-\rho_m(x))\le \varphi(\ell-u/2),\qquad x\in\Lambda_r.
    \end{align*}

    Combining \eqref{eq:SA-unshifted-density}, \eqref{eq:SA-shift-Kac-Rice} and \eqref{eq:varphi upper bound} therefore yields
    \begin{equation}\label{eq:SA-core}
        I_1=\mathbb E\mu_{\operatorname{SA}}(\Lambda_r;f+\rho_m)-\mathbb E\mu_{\operatorname{SA}}(\Lambda_r;f)
        \leq-\frac{3c_{\operatorname{SA}}(\ell)}{4}(2r)^d.
    \end{equation}

    For $I_2$, the number of cubes in $I_s\setminus I_r$ is
    \begin{align*}
        \left|I_s\setminus I_r\right|=(2s)^d-(2r)^d
        \leq C_{d,R}(m^{d-1}+1).
    \end{align*}

    Using the uniform moment bound \eqref{eq:SA-shift-moments},
    \begin{equation}\label{eq:SA-transition}
        \left|I_2\right|\leq C_{u}\left|I_s\setminus I_r\right|\leq C_{u,R}(m^{d-1}+1).
    \end{equation}

    For the exterior term $I_3$, denote
    \begin{align}\label{eq:SA-distance-definition}
        d_m(w):=\operatorname{dist}(w,\Lambda_{m+1}),
    \end{align}
    where $\operatorname{dist}$ is Euclidean distance. 
    Since $B_w=w+[0,1]^d$,
    \begin{align}\label{eq:SA-distance-comparison}
        L_{m,w}\leq d_m(w)\le\sqrt d\,(L_{m,w}+1).
    \end{align}

    For $w\in I_n\setminus I_s$ we have $L_{m,w}\ge R$.
    Thus, using \eqref{eq:SA-summable-exterior} and \citet[Lemma~2.4]{McAuley2026a} with $n=m+1$ and $r=R/2$,
    \begin{align}\label{eq:SA-exterior}
        \left|I_3\right|\leq C_u\sum_{w:L_{m,w}\ge R}L_{m,w}^{-\gamma_{\operatorname{SA}}}\leq C_{u,R}(m^{d-1}+1),
    \end{align}
    since $\gamma_{\operatorname{SA}}>d$.

    Combining \eqref{eq:SA-three-regions}--\eqref{eq:SA-exterior},
    \begin{align*}
        \mathbb E\mu_{\operatorname{SA}}(\Lambda_n;f+\rho_m)-\mathbb E\mu_{\operatorname{SA}}(\Lambda_n;f)
        \leq-\frac{3c_{\operatorname{SA}}(\ell)}{4}(2m-2R)^d+C_{u,R}(m^{d-1}+1),
    \end{align*}
    uniformly for $n\geq m+1+R$.
    For fixed $u,R$, choose $m$ sufficiently large, then 
    \begin{align*}
        \mathbb E\mu_{\operatorname{SA}}(\Lambda_n;f+\rho_m)-\mathbb E\mu_{\operatorname{SA}}(\Lambda_n;f)<0.
    \end{align*}

    For this fixed $m$, let $n\to\infty$, Lemma~\ref{lem:Gaussian-response} yields $R_{\operatorname{SA}}(h_m)<0$,
    thus $\sigma_{\operatorname{SA}}^2(\ell)>0$.
    The variance convergence and finiteness follow from \citet[Theorem~1.4]{McAuley2026a}.
\end{proof}

\subsection{Euler characteristic variance}\label{subsec:EC-variance-positive}
In the sequel we will prove Proposition~\ref{prop:EC-positive}. 
First we give a technical lemma. For EC, it turns the perturbation $\rho_m$ into some constant $u$.

\begin{lemma}\label{lem:EC-local-lift}
    Assume (M1)--(M4) of Assumption~\ref{ass:McAuley}, fix $\ell<\ell_c$ and $u>0$,
    and let $p_j\in C^2(\mathbb R^d)$ be deterministic functions such that
    \begin{align*}
        0\leq p_j\leq u,\qquad p_j\longrightarrow u\quad\text{in }C^2_{\mathrm{loc}}(\mathbb R^d).
    \end{align*}

    Let $F$ be a fixed finite union of closed faces of a unit cube, possibly the whole cube or the empty set. 
    If, for some $r>1$, $\sup_j\mathbb E|\mu_{\operatorname{EC}}(F;f+p_j)|^r<\infty$, then
    \begin{align*}
        \mathbb E\left|\mu_{\operatorname{EC}}(F;f+p_j)-\mu_{\operatorname{EC}}(F;f+u)\right|\longrightarrow0,
        \quad j\rightarrow\infty.
    \end{align*}
\end{lemma}

\begin{proof}
    Let $Q$ be the unit cube containing $F$. We apply \citet[Lemma~3.1]{McAuley2026a} with
    \begin{align*}
        g:=f+u,\quad q_j:=p_j-u,\quad g+tq_j=f+u+t(p_j-u),\quad 0\le t\le1,
    \end{align*}

    Since $\ell-u$ is almost surely a regular value of $f$ on every cubical stratum of $Q$, 
    $\ell$ is a regular value of $g=f+u$, see \citet[Lemma~4.10]{McAuley2026a}.
    Hence, with probability 1, for every stratum $S$ of $Q$,
    \begin{align*}
        \{x\in S:g(x)=\ell,\ \nabla_S g(x)=\mathbf{0}\}=\varnothing.
    \end{align*}

    Since $Q$ has only finitely many closed cubical strata, 
    compactness and stratified regularity give a random $\delta>0$ such that
    \begin{align*}
        |g(x)-\ell|+|\nabla_Sg(x)|\ge\delta
    \end{align*}
    for every stratum $S$ of $Q$ and every $x\in S$.

    Therefore, since $p_j\to u$ in $C^2_{\mathrm{loc}}$, 
    for sufficiently large $j$, uniformly in $x\in S\subset Q, t\in[0,1]$ we have,

    \begin{align*}
        |g(x)+tq_j(x)-\ell|+|\nabla_S(g+tq_j)(x)|
        &\geq|g(x)-\ell|+|\nabla_Sg(x)|-t|q_j(x)|-t|\nabla_Sq_j(x)|\\
        &\geq \frac{\delta}{2}.
    \end{align*}

    In other words, $(g,q_j)=(f+u,p_j-u)$ is locally stable on $Q$ in the sense of \citet[Lemma~3.1]{McAuley2026a}.
    Consequently, for all sufficiently large $j$, that lemma gives
    a stratified isotopy $H_{j,t}$, $0\le t\le1$, with
    $H_{j,0}=\operatorname{id}$ and for the fixed union of faces $F$,
    \begin{align}\label{eq:EC-lift-isotopy}
        H_{j,t}\bigl(\{f+u\ge\ell\}\cap F\bigr)=\{f+u+t(p_j-u)\ge\ell\}\cap F.
    \end{align}

    In particular, at $t=1$,
    \begin{align}\label{eq:EC-lift-endpoint}
       H_{j,1}\bigl(\{f+u\ge\ell\}\cap F\bigr)=\{f+p_j\ge\ell\}\cap F. 
    \end{align}

    Apply Lemma~\ref{lem:cubical-excursion-components} to each face $F_i$ in $F$,
    with $f+u$ restricted to lower dimensions, 
    the set $\{f+u\geq\ell\}\cap F$ has finitely many connected components.
    Equation~\eqref{eq:EC-lift-endpoint} therefore gives a one-to-one
    correspondence $C\mapsto H_{j,1}(C)$ between the components of
    $\{f+u\geq\ell\}\cap F$ and those of
    $\{f+p_j\geq\ell\}\cap F$.

    Stratified regularity allows us to choose, in each nonempty component $C$, 
    an inner point $x_C$ such that $f(x_C)+u>\ell$. 
    Since there are only finitely many $x_C$ and $p_j\to u$ locally uniformly, 
    for all sufficiently large $j$,
    \begin{align*}
        f(x_C)+u+t\bigl(p_j(x_C)-u\bigr)>\ell,\qquad 0\leq t\leq1,
    \end{align*}
    simultaneously for every $C$. 
    Combining with \eqref{eq:EC-lift-isotopy}, the continuous curve $t\mapsto H_{j,t}^{-1}(x_C)$ lies in $C$,
    and hence $x_C\in H_{j,1}(C)$.

    Now we claim that
    \begin{equation}\label{eq:EC-lift-label}
        C\subseteq\{f+u\geq\ell\}_\infty\quad\Longleftrightarrow\quad H_{j,1}(C)\subseteq\{f+p_j\geq\ell\}_\infty.
    \end{equation}

    We use only the inclusions
    \begin{align}\label{inclusion}
        \{f\geq\ell\}\subseteq\{f+p_j\geq\ell\}\subseteq\{f+u\geq\ell\}.
    \end{align}

    \medskip
    \noindent\textbf{\textit{Step 1.}}
    Suppose that $H_{j,1}(C)$ is contained in the unbounded component of $\{f+p_j\geq\ell\}$. 
    This component contains $x_C$ and, by \eqref{inclusion}, is a connected unbounded subset of $\{f+u\geq\ell\}$. 
    Thus the component of $\{f+u\geq\ell\}$ containing $x_C$ is unbounded. In other words,
    \begin{align*}
        C\subseteq\{f+u\geq\ell\}_\infty.
    \end{align*}

    \medskip
    \noindent\textbf{\textit{Step 2.}}
    Suppose that $C\subseteq\{f+u\geq\ell\}_\infty$. 
    Both $\{f\geq\ell\}$ and $\{f+u\geq\ell\}$ have a unique unbounded component, and the former set is contained in the latter. 
    Hence both $x_C$ and $\{f\geq\ell\}_\infty$ lie in $\{f+u\geq\ell\}_{\infty}$. 
    
    Choose $y\in\{f\geq\ell\}_\infty$ with $f(y)>\ell$. 
    By regularity, there is a compact path $\gamma\subseteq\{f+u>\ell\}$ joining $x_C$ to $y$. 
    Local uniform convergence of $p_j$ to $u$ gives, for all sufficiently large $j$,
    \begin{align*}
        \inf_{z\in\gamma}\left(f(z)+p_j(z)-\ell\right)\geq\inf_{z\in\gamma}\left(f(z)+u-\ell\right)-\sup_{z\in\gamma}\left(u-p_j(z)\right)>0,
    \end{align*}
    so $\gamma\subseteq\{f+p_j>\ell\}\subseteq\{f+p_j\geq\ell\}$.
    As $y\in\{f\geq\ell\}_\infty\subset\{f+p_j\geq\ell\}_{\infty}$, we must have $\gamma\subset\{f+p_j\geq\ell\}_{\infty}$.
    Especially, $x_C\in\{f+p_j\geq\ell\}_{\infty}$.
    Since $x_C\in H_{j,1}(C)$, this proves the implication from left to right in \eqref{eq:EC-lift-label}.

    There are only finitely many components $C$, so
    \eqref{eq:EC-lift-label} holds for all of them simultaneously when
    $j$ is sufficiently large. Together with
    \eqref{eq:EC-lift-endpoint}, it gives
    \begin{align*}
        H_{j,1}\bigl(F\cap\{f+u\geq\ell\}_\infty\bigr)=F\cap\{f+p_j\geq\ell\}_\infty.
    \end{align*}

    Hence topological invariance of Euler characteristic yields $\mu_{\operatorname{EC}}(F;f+p_j)$ almost surely converge to $\mu_{\operatorname{EC}}(F;f+u)$.
    Combining with $\sup_j\mathbb E|\mu_{\operatorname{EC}}(F;f+p_j)|^r<\infty$, 
    the almost sure convergence to zero implies
    \begin{align*}
        \mathbb E\left|\mu_{\operatorname{EC}}(F;f+p_j)-\mu_{\operatorname{EC}}(F;f+u)\right|\longrightarrow0.
    \end{align*}
\end{proof}

\begin{proof}[Proof of Proposition~\ref{prop:EC-positive}]
    We prove that  $R_{\operatorname{EC}}(h_m)\ne0$, 
    and then apply Lemma~\ref{lem:Gaussian-response}.

    \medskip\noindent
    \textbf{\textit{Step 1.}}
    By Lemma~\ref{lem:EC-cavity-nonvanishing} and
    \citet[Theorem~1.3]{McAuley2026a}, choose $u>0$ such that
    \begin{align*}
        \Delta:=c_{\operatorname{EC}}(\ell-u)-c_{\operatorname{EC}}(\ell)\ne0.
    \end{align*}

    Indeed, if $c_{\operatorname{EC}}(\ell)=0$, then
    Lemma~\ref{lem:EC-cavity-nonvanishing} gives
    $c_{\operatorname{EC}}(\ell-u)\ne0$ for all sufficiently large $u$.
    If $c_{\operatorname{EC}}(\ell)\ne0$, then
    $c_{\operatorname{EC}}(\ell-u)\to0$ as $u\to\infty$ by
    \citet[Theorem~1.3]{McAuley2026a}, so the same conclusion follows.
    Fix this $u$ and recall $\Gamma_{n,w}$ in \eqref{eq:EC-cube-geometry}, $h_m,\rho_m$ in \eqref{eq:smooth-white-noise-shift}.

    By the proof of \citet[Theorem~1.3]{McAuley2026a}, with stationarity and symmetry of the field $f$,
    \begin{equation}\label{eq:EC-cell-mean}
        \mathbb E\bigl[\mu_{\operatorname{EC}}(B_{\mathbf{0}};f+t)-\mu_{\operatorname{EC}}(\Gamma_{n,\mathbf{0}};f+t)\bigr]
        =c_{\operatorname{EC}}(\ell-t),\qquad t\in\{0,u\},\quad n\geq2.
    \end{equation}
    
    For $w\in I_r, r<n$, $\Gamma_{n,w}=w+\Gamma_{n,\mathbf{0}}$.
    Hence, by stationarity,
    \begin{equation}\label{eq:EC-constant-cell-response}
        \mathbb E\bigl[\xi_{n,w}(f+u)-\xi_{n,w}(f)\bigr]
        =c_{\operatorname{EC}}(\ell-u)-c_{\operatorname{EC}}(\ell)
        =\Delta.
    \end{equation}

    \medskip\noindent
    \textbf{\textit{Step 2.}}
    For an integer $R\ge1$, denote
    \begin{equation}\label{eq:EC-uniform-interior}
        \varepsilon_R:=
        \sup_{\substack{m\in\mathbb N\colon m>2R}}
        \ \sup_{\substack{n\in\mathbb N\colon n\ge m+1+R}}
        \ \sup_{w\in I_{m-R}}\mathbb E\left|\xi_{n,w}(f+\rho_m)-\xi_{n,w}(f+u)\right|.
    \end{equation}

    The uniform moment bounds in \eqref{eq:EC-shift-moments}
    ensure that $\varepsilon_R<\infty$.
    For each integer $R\ge2$, choose $m_R>2R, n_R\geq m_R+1+R, w_R\in I_{m_R-R}$ such that
    \begin{align}\label{epsilon_R bound}
        0\le\varepsilon_R&\le\mathbb E\left|\xi_{n_R,w_R}(f+\rho_{m_R})-\xi_{n_R,w_R}(f+u)\right|+\frac1R.
    \end{align}

    Define $p_R(z):=\rho_{m_R}(w_R+z)$.
    By \eqref{eq:smooth-shift-regularity} and \eqref{eq:smooth-shift-core},
    \begin{align*}
        0\le p_R\le u,\quad p_R\longrightarrow u\quad\text{in }C^2_{\mathrm{loc}}(\mathbb R^d).
    \end{align*}

    Moreover, by \citet[Lemma~4.12]{McAuley2026a}, stationarity and \eqref{eq:EC-shift-moments},
    \begin{align*}
        \sup_{R\ge2}\mathbb E|\mu_{\operatorname{EC}}(F;f+p_R)|^{k-1}\le C_u,\qquad F\in\{B_{\mathbf{0}},\Gamma_{2,\mathbf{0}}\}.
    \end{align*}

    Thus Lemma~\ref{lem:EC-local-lift}, applied to these two fixed sets, gives
    \begin{align*}
        \mathbb E\left|\xi_{2,\mathbf{0}}(f+p_R)-\xi_{2,\mathbf{0}}(f+u)\right|=
        \mathbb E\left|\xi_{n_R,w_R}(f+\rho_{m_R})-\xi_{n_R,w_R}(f+u)\right|\longrightarrow0.
    \end{align*}

    Since $\Gamma_{n_R,w_R}=w_R+\Gamma_{2,\mathbf{0}}$, when $R\rightarrow\infty$, \eqref{epsilon_R bound} yields
    \begin{align*}
        \varepsilon_R\longrightarrow0.
    \end{align*}

    \medskip\noindent
    \textbf{\textit{Step 3.}}
    Fix an integer $R\geq 2$ such that \eqref{eq:EC-summable-exterior} applies whenever $L_{m,w}\ge R$,
    and 
    \begin{align*}
        \epsilon_R\leq \frac{\left|\Delta\right|}{4}.
    \end{align*}

    For integers $m>2R$ and $n\ge m+1+R$, let 
    \begin{align*}
        r=m-R,\qquad s=m+1+R.
    \end{align*}

    Using the same argument as \eqref{eq:EC-exact-sum} for the new field $f+\rho_m$, we have the following decomposition
    \begin{equation}\label{eq:EC-three-regions}
        \mathbb E\mu_{\operatorname{EC}}(\Lambda_n;f+\rho_m)-\mathbb E\mu_{\operatorname{EC}}(\Lambda_n;f)
        =I_1+I_2+I_3,
    \end{equation}
    where
    \begin{align*}
        I_1&:=\sum_{w\in I_r}\mathbb E\bigl[\xi_{n,w}(f+\rho_m)-\xi_{n,w}(f)\bigr],\\
        I_2&:=\sum_{w\in I_s\setminus I_r}\mathbb E\bigl[\xi_{n,w}(f+\rho_m)-\xi_{n,w}(f)\bigr],\\
        I_3&:=\sum_{w\in I_n\setminus I_s}\mathbb E\bigl[\xi_{n,w}(f+\rho_m)-\xi_{n,w}(f)\bigr].
    \end{align*}

    For $I_1$, \eqref{eq:EC-constant-cell-response} and \eqref{eq:EC-uniform-interior} give
    \begin{align}\label{eq:EC-core-response}
        |I_1-\Delta(2r)^d|
        &=\left|\sum_{w\in I_r}\mathbb E\bigl[\xi_{n,w}(f+\rho_m)-\xi_{n,w}(f+u)\bigr]\right|\notag\\
        &\leq\sum_{w\in I_r}\mathbb E\left|\xi_{n,w}(f+\rho_m)-\xi_{n,w}(f+u)\right|\notag\\
        &\leq\varepsilon_R(2r)^d.
    \end{align}

    For $I_2$, the uniform moments in
    \eqref{eq:EC-shift-moments} imply
    \begin{align}\label{eq:EC-transition-response}
        |I_2|&\leq\sum_{w\in I_s\setminus I_r}\bigl(\mathbb E|\xi_{n,w}(f+\rho_m)|+\mathbb E|\xi_{n,w}(f)|\bigr)\notag\\
        &\leq C_u\bigl((2s)^d-(2r)^d\bigr)\notag\\
        &\leq C_{u,R}(m^{d-1}+1).
    \end{align}

    For $I_3$, recall \eqref{eq:EC-summable-exterior}, let $\gamma=\gamma_{\operatorname{EC}}>d$. 
    Since $w\in I_n\setminus I_s$ implies $L_{m,w}\geq R$, 
    the same argument as in \eqref{eq:SA-distance-definition}--\eqref{eq:SA-exterior} gives
    \begin{align}\label{eq:EC-exterior-response}
        |I_3|\leq C_u\sum_{w:L_{m,w}\geq R}L_{m,w}^{-\gamma}
        \leq C_{u,R}(m^{d-1}+1),
    \end{align}
    uniformly in $n\geq s$.

    \medskip\noindent
    \textbf{\textit{Step 4.}}
    Combining \eqref{eq:EC-three-regions}, \eqref{eq:EC-core-response}, \eqref{eq:EC-transition-response} and \eqref{eq:EC-exterior-response}, 
    we obtain that for $n\geq s$,
    \begin{align}\label{eq:EC-response-asymptotic}
        \left|\mathbb E\mu_{\operatorname{EC}}(\Lambda_n;f+\rho_m)-\mathbb E\mu_{\operatorname{EC}}(\Lambda_n;f)-\Delta(2r)^d\right|
        \leq\varepsilon_R(2r)^d+C_{u,R}(m^{d-1}+1).
    \end{align}

    Since $u,R$ are fixed and $r=m-R$, choose $m>2R$ such that 
    \begin{align*}
        C_{u,R}(m^{d-1}+1)\leq \frac{\left|\Delta\right|}{4}(2r)^d
    \end{align*}
    
    Finally let $n\rightarrow\infty$, \eqref{eq:EC-response-asymptotic} and Step 2 show that, for a suitable $R>0$,
    \begin{align*}
        |R_{\operatorname{EC}}(h_m)|\geq |\Delta|(2r)^d-\left|R_{\operatorname{EC}}(h_m)-\Delta(2r)^d\right|\geq\frac{|\Delta|}{2}(2r)^d>0.
    \end{align*}

    Lemma~\ref{lem:Gaussian-response} now implies $\sigma_{\operatorname{EC}}^2(\ell)>0$.
    Variance convergence and finiteness follow from \citet[Theorem~1.4]{McAuley2026a}.
\end{proof}

\subsection{Low-level Euler density}\label{subsec:EC-low-density}
Recall the law of large numbers proved in \citet{McAuley2026a} that 
\begin{align*}
    c_{\operatorname{EC}}(\ell)=\lim_{n\rightarrow\infty}\frac{\mu_{\operatorname{EC}}(\Lambda_n;f)}{\left|\Lambda_n\right|}.
\end{align*}

The main result of this subsection is the following lemma.

\begin{lemma}\label{lem:EC-cavity-nonvanishing}
Assume (M1)--(M4) of Assumption~\ref{ass:McAuley}.  
There exists $u_0<\infty$ such that
\begin{align*}
    c_{\operatorname{EC}}(-u)\ne0,\qquad u\ge u_0.
\end{align*}
\end{lemma}

\begin{proof}
    Fix $a_0<\min\{\ell_c,-\ell_c\}, -u<a_0$.
    Every component $C$ of the sublevel set $L_{-u}:=\{f\leq -u\}$ is then bounded, as $f\stackrel d=-f$.
    If $V_\infty(C)$ denotes the unbounded component of $\mathbb R^d\setminus C$, define its filling
    \footnote{For example, if $C=\{x\in\mathbb R^2:1\leq |x|\leq2\}$, then $\operatorname{Fill}(C)=\{x\in\mathbb R^2:|x|\leq2\}$.} by
    \begin{align*}
        \operatorname{Fill}(C):=\mathbb R^d\setminus V_\infty(C).
    \end{align*}

    Then we have 
    \begin{align*}
        C\subseteq\operatorname{Fill}(C)\subseteq\operatorname{conv}(C),
        \qquad
        \operatorname{diam}\operatorname{Fill}(C)=\operatorname{diam}C.
    \end{align*}

    Filled components are either disjoint or nested.  
    Let $\mathscr D_{-u}$ be the collection of maximal filled components. 
    \footnote{For example, take two components $C_1=B_2(\mathbb{R}^2)\setminus B_1(\mathbb{R}^2), C_2=B_4(\mathbb{R}^2)\setminus B_3(\mathbb{R}^2)$,
    then $\operatorname{Fill}(C_1)\subset\operatorname{Fill}(C_2)$, and we only take $\operatorname{Fill}(C_2)\in\mathscr D_{-u}$.} 
    Additionally, Lemma~\ref{lem:EC-cavity-large-tail} ensures that maximal filled components exists with probability 1.
    Then we claim that
    \begin{equation}\label{eq:EC-cavity-decomposition}
        \{f\geq -u\}_\infty
        =\mathbb R^d\setminus\bigcup_{D\in\mathscr D_{-u}}\operatorname{int}D.
    \end{equation}

    Indeed, the right-hand side is an unbounded subset of $\{f\geq -u\}$ and contains $\{f\geq -u\}_\infty$. 
    Take $x_0$ in the right-hand side, then $x_0\in\{f\geq -u\}$.
    If $x_0$ is contained in a bounded component $C_0\subset\{f\geq -u\}_{<\infty}$,
    then $C_0$ must be surrounded by some component of $\{f\leq -u\}$, which must contained in some filling $D\in\mathscr D_{-u}$. 
    Thus $x_0\in \{f\geq -u\}_\infty$. Therefore \eqref{eq:EC-cavity-decomposition} holds.

    Recall the proof of \citet[Theorem~1.3]{McAuley2026a}, we have 
    \begin{equation}\label{eq:EC-cavity-cell-mean}
        c_{\operatorname{EC}}(-u)=\mathbb E\psi_{-u}(B_{\mathbf 0}),
    \end{equation}
    where 
    \begin{align*}
        \psi_{-u}(B_v):=\operatorname{EC}(\{f\geq -u\}_\infty\cap B_v)-\operatorname{EC}(\{f\geq -u\}_\infty\cap\partial_{\mathrm{LL}}B_v),   
    \end{align*}
    and
    \begin{align*}
        \partial_{\mathrm{LL}}B_v:=\{x\in\partial B_v:x_i=v_i\text{ for some }i\}.
    \end{align*}
    By \eqref{eq:EC-cavity-decomposition}, local finiteness of $\mathscr D_{-u}$ (see Lemma~\ref{lem:EC-cavity-large-tail}), 
    and Euler additivity, for $E=\partial_{\mathrm{LL}}B_v$ or $B_v$,
    \begin{equation}\label{eq:EC-cavity-additivity-on-E}
        \operatorname{EC}(\{f\geq -u\}_\infty\cap E)
        =\operatorname{EC}(E)+\sum_{D\in\mathscr D_{-u}}\bigl(\operatorname{EC}(\partial D\cap E)-\operatorname{EC}(D\cap E)\bigr).
    \end{equation}

    Subtracting \eqref{eq:EC-cavity-additivity-on-E} for $E=\partial_{\mathrm{LL}}B_v$ and $E=B_v$,
    and using $\operatorname{EC}(B_v)=\operatorname{EC}(\partial_{\mathrm{LL}}B_v)=1$, gives
    \begin{equation}\label{eq:EC-cavity-cell-decomposition}
        \psi_{-u}(B_v)=\sum_{D\in\mathscr D_{-u}}w_v(D),
    \end{equation}
    where
    \begin{align*}
        w_v(D):=\operatorname{EC}(\partial D\cap B_v)-\operatorname{EC}(D\cap B_v)
        -\operatorname{EC}(\partial D\cap\partial_{\mathrm{LL}}B_v)+\operatorname{EC}(D\cap\partial_{\mathrm{LL}}B_v).
    \end{align*}

    Fix $D\in\mathscr D_{-u}$.
    Since $D$ is compact, only finitely many of the following terms are
    nonzero.
    The cubical additivity formula, applied first to $D$ and then
    to $\partial D$, gives
    \begin{align*}
        \sum_{v\in\mathbb Z^d}
        \bigl(\operatorname{EC}(A\cap B_v)
        -\operatorname{EC}(A\cap\partial_{\mathrm{LL}}B_v)\bigr)
        =\operatorname{EC}(A),
        \qquad A\in\{D,\partial D\}.
    \end{align*}

    Thus
    \begin{align*}
        \sum_{v\in\mathbb Z^d}w_v(D)=\operatorname{EC}(\partial D)-\operatorname{EC}(D)=(-1)^{d-1}\operatorname{EC}(D).
    \end{align*}

    Now we choose a unique anchor point $z(D)\in\mathbb{Z}^d$ for each $D\in\mathscr{D}_{-u}$,
    where we let the minimum point of $f$ in $D$ contains in $B_{z(D)}$.
    If the minimum point is not unique, then $z(D)$ is the minimum one in lexicographic order.
    Then $z(D+k)=z(D)+k$ for every $k\in\mathbb Z^d$.
    Therefore we have
    \begin{align*}
        c_{\operatorname{EC}}(-u)
        &=\sum_{z\in\mathbb Z^d}\mathbb E\sum_{\substack{D\in\mathscr D_{-u}\\z(D)=z}}w_{\mathbf 0}(D)\\
        &=\sum_{z\in\mathbb Z^d}\mathbb E\sum_{\substack{D\in\mathscr D_{-u}\\z(D)=\mathbf 0}}w_{-z}(D)\\
        &=\mathbb E\sum_{\substack{D\in\mathscr D_{-u}\\z(D)=\mathbf 0}}\sum_{v\in\mathbb Z^d}w_v(D)\\
        &=(-1)^{d-1}\mathbb E\sum_{\substack{D\in\mathscr D_{-u}\\z(D)=\mathbf 0}}\operatorname{EC}(D),
    \end{align*}
    where Lemma~\ref{lem:EC-cavity-absolute-summability} justifies the exchange of expectation and summation.

    Equivalently,
    \begin{equation}\label{eq:EC-cavity-mass-transport}
        (-1)^{d-1}c_{\operatorname{EC}}(-u)
        =\mathbb E\sum_{\substack{D\in\mathscr D_{-u}\\z(D)=\mathbf 0}}\operatorname{EC}(D).
    \end{equation}

    Let $N_{\mathrm{bad}}(C)$ be the number of critical points of $f$ in $C$ that are not local minima.  
    Define 
    \begin{align*}
        \rho_u&:=\mathbb E\#\{D\in\mathscr D_{-u}:z(D)=\mathbf 0\},\\
        \nu_{\mathrm{bad}}(u)&:=\mathbb E\#\bigl\{x\in[0,1)^d:\nabla f(x)=\mathbf 0,\ f(x)\le-u,\ \nabla^2f(x)\not\succ0\bigr\}.
    \end{align*}

    Denote $Q_v=v+[0,1)^d, C(D)$ be the component of $\{f\leq -u\}$ whose filling is $D$.  
    Tonelli's theorem and stationarity give
    \begin{align*}
        \mathbb E\sum_{\substack{D\in\mathscr D_{-u}\\z(D)=\mathbf 0}}N_{\mathrm{bad}}(C(D))
        &=\sum_{v\in\mathbb Z^d}\mathbb E\sum_{\substack{D\in\mathscr D_{-u}\\z(D)=\mathbf 0}}N_{\mathrm{bad}}(C(D)\cap Q_v)\\
        &=\sum_{v\in\mathbb Z^d}\mathbb E\sum_{\substack{D\in\mathscr D_{-u}\\z(D)=-v}}N_{\mathrm{bad}}(C(D)\cap Q_{\mathbf 0})\\
        &=\mathbb E\sum_{D\in\mathscr D_{-u}}N_{\mathrm{bad}}(C(D)\cap Q_{\mathbf 0})\leq\nu_{\mathrm{bad}}(u).
    \end{align*}

    Consequently, \eqref{eq:EC-cavity-mass-transport}, Lemma~\ref{lem:EC-cavity-topological-bound} and Lemma~\ref{lem:EC-cavity-bad-intensity} imply
    \begin{equation}\label{eq:EC-cavity-density-lower-bound}
        (-1)^{d-1}c_{\operatorname{EC}}(-u)\geq\rho_u-\nu_{\mathrm{bad}}(u)\geq\rho_u-Ce^{-(1/2+\delta)u^2}.
    \end{equation}

    It remains to obtain a lower bound on $\rho_u$.  
    Let
    \begin{align*}
        v_u:=\mathbb P(\mathbf 0\notin\{f\geq-u\}_\infty).
    \end{align*}

    Since $f(\mathbf 0)<-u$ implies $\mathbf 0\notin\{f\geq-u\}_\infty$,
    \begin{equation}\label{eq:EC-cavity-volume-lower-bound}
        v_u\geq \Phi(-u),
    \end{equation}
    where $\Phi$ is the standard Gaussian distribution function.
    
    By stationarity of the field,
    \begin{align*}
        \mathbb P\bigl(\exists D\in\mathscr D_{-u}:
        \mathbf 0\in\operatorname{int}D,\ \operatorname{diam}D\le R\bigr)
        &\leq\sum_{\substack{v\in\mathbb Z^d\\|v|_\infty\le R+1}}\mathbb P\bigl(\exists D\in\mathscr D_{-u}:z(D)=v\bigr)\\
        &\leq\sum_{\substack{v\in\mathbb Z^d\\|v|_\infty\le R+1}}\mathbb{P}\left(\#\{D\in\mathscr D_{-u}:z(D)=v\}\geq 1\right)\\
        &\leq\sum_{\substack{v\in\mathbb Z^d\\|v|_\infty\le R+1}}\mathbb E\#\{D\in\mathscr D_{-u}:z(D)=v\}\\
        &\leq C_dR^d\rho_u.
    \end{align*}

    Therefore we have
    \begin{equation}\label{eq:EC-cavity-volume-decomposition}
        v_u\le C_dR^d\rho_u+
        \mathbb P\left(\exists D\in\mathscr D_{-u}:
        \mathbf 0\in D,\ \operatorname{diam}D>R\right).
    \end{equation}

    Suppose now that $c_{\operatorname{EC}}(-u)=0$. 
    Combining \eqref{eq:EC-cavity-density-lower-bound}--\eqref{eq:EC-cavity-volume-decomposition} and Lemma~\ref{lem:EC-cavity-large-tail}
    we obtain
    \begin{equation}\label{eq:EC-cavity-contradiction-bound}
        \Phi(-u)\le CR^de^{-(1/2+\delta)u^2}+C_0e^{-c_0R}.
    \end{equation}

    Choose $R=Au^2$ with $c_0A>1/2$.  
    Since $\Phi(-u)\asymp u^{-1}e^{-u^2/2}$, division by $\Phi(-u)$ in \eqref{eq:EC-cavity-contradiction-bound} gives
    \begin{align*}
        1\le Cu^{2d+1}e^{-\delta u^2}+C_0u e^{-(c_0A-1/2)u^2}.
    \end{align*}

    Let $u\rightarrow\infty$, then this leads to a contradiction. 
    Therefore, there exist $u_0<\infty$, for any $u>u_0$, $c_{\operatorname{EC}}(-u)\not=0$.
\end{proof}

Finally, we state several lemmas applied in the proof of Lemma~\ref{lem:EC-cavity-nonvanishing}.

\begin{lemma}\label{lem:EC-cavity-absolute-summability}
    With the notation of Lemma~\ref{lem:EC-cavity-nonvanishing}, we have
    \begin{align*}
        \mathbb E\sum_{\substack{D\in\mathscr D_{-u}\\z(D)=\mathbf 0}}
        \sum_{v\in\mathbb Z^d}|w_v(D)|<\infty.
    \end{align*}
\end{lemma}

\begin{lemma}
    \label{lem:EC-cavity-topological-bound}
    Let $g\in C^2(\mathbb R^d)$, let $a$ be a regular value of $g$, and let
    $C$ be a compact connected component of $\{g\le a\}$.  Suppose that all
    critical points of $g$ in $C$ are nondegenerate, and 
    $D=\operatorname{Fill}(C)$.  If $N_{\mathrm{bad}}(C)$ is the number of
    critical points of $g$ in $C$ that are not local minima, then
    \begin{equation}\label{eq:EC-cavity-topological-bound}
        \operatorname{EC}(D)\ge1-N_{\mathrm{bad}}(C).
    \end{equation}
\end{lemma}

\begin{lemma}
    \label{lem:EC-cavity-bad-intensity}
    Assume (M1)--(M4) of Assumption~\ref{ass:McAuley}, and define
    \begin{align*}
        \nu_{\mathrm{bad}}(u)
        :=\mathbb E\#\bigl\{x\in[0,1)^d:\nabla f(x)=\mathbf 0,
        \ f(x)\le-u,\ \nabla^2f(x)\not\succ0\bigr\}.
    \end{align*}
    There exist constants $C,\delta>0$ such that, for all sufficiently large
    $u$,
    \begin{equation}\label{eq:EC-cavity-bad-intensity}
        \nu_{\mathrm{bad}}(u)\le Ce^{-(1/2+\delta)u^2}.
    \end{equation}
\end{lemma}

\begin{lemma}\label{lem:EC-cavity-large-tail}
    Assume (M1)--(M4) of Assumption~\ref{ass:McAuley}, and fix $a_0<-\ell_c$.
    For $u$ satisfying $-u\le a_0$, let $\mathscr F_{-u}$ be the family of
    all fillings of the components of $\{f\le-u\}$.
    There exist constants $C_0,c_0>0$, independent of $u$ and $R$, such that
    for every $R\ge1$,
    \begin{equation}\label{eq:EC-cavity-large-tail}
        \mathbb P\left(\exists D\in\mathscr F_{-u}:
        \mathbf 0\in D,\ \operatorname{diam}D>R\right)
        \le C_0e^{-c_0R}.
    \end{equation}

    Further, for each fixed such $u$, 
    almost surely every element of $\mathscr F_{-u}$ is contained in a maximal filling.

    Additionally, the family $\mathscr D_{-u}$ of maximal fillings is also almost surely locally finite: 
    every compact subset of $\mathbb R^d$ meets only finitely many members of $\mathscr D_{-u}$.
\end{lemma}

The proofs of Lemmas~\ref{lem:EC-cavity-absolute-summability}--\ref{lem:EC-cavity-large-tail} are as following.

\begin{proof}[Proof of Lemma~\ref{lem:EC-cavity-absolute-summability}]
    Let $N_v=N_{\operatorname{Crit}}(B_v,f)$ denote the total number of critical points on the open cubical strata of $B_v$, 
    including its vertices.  
    Fix a finite subfamily $\mathscr A\subseteq\mathscr D_{-u}$ and set
    \begin{align}\label{eq:EC-cavity-KLH}
        K:=\bigcup_{D\in\mathscr A}D,\qquad
        L:=K\cap\{f\leq -u\},\qquad
        H:=\operatorname{int}K\cap\{f\geq -u\}.
    \end{align}

    Here and below, $F$ may be a closed cubical face of $B_v$, with
    $F=B_v$ allowed, or a finite union of such faces.  Statements involving
    $\dim F$ are first applied to a single face and then extended to finite
    unions by inclusion--exclusion.

    The decompositions
    \begin{align}\label{eq:EC-cavity-KLH-decomposition}
        K=L\cup H,\qquad L\cap H=\partial H,\qquad
        \partial L=\partial K\bigsqcup\partial H
    \end{align}
    and Euler additivity imply
    \begin{align}\label{eq:EC-cavity-KLH-euler}
        \operatorname{EC}(\partial K\cap F)-\operatorname{EC}(K\cap F)
        =\operatorname{EC}(\partial L\cap F)-\operatorname{EC}(L\cap F)-\operatorname{EC}(H\cap F).
    \end{align}

    The set $L$ is a union of components of $\{f\leq -u\}$, 
    whereas $H$ is a union of bounded components of $\{f\geq -u\}$.  
    Hence, for every closed cubical face $F$ of $B_v$, allowing $F=B_v$, \citet[Lemma~4.12]{McAuley2026a}, 
    applied to $-f$ for $L$ and to $f$ for $H$, gives
    \begin{align}\label{eq:EC-cavity-excursion-face-bound}
        |\operatorname{EC}(L\cap F)|+|\operatorname{EC}(H\cap F)|\leq C_dN_v.
    \end{align}
    
    Let $k=\dim F$ and $M=L\cap F$.  
    For a zero-dimensional face $F$, $\partial L\cap F=\varnothing$.
    For $1\leq k\leq d$, Euler additivity and the boundary Euler formula give
    \begin{align}\label{eq:EC-cavity-L-boundary-recursion}
        \operatorname{EC}(\partial L\cap F)
        =\bigl(1-(-1)^k\bigr)\operatorname{EC}(L\cap F)-\operatorname{EC}(L\cap\partial F)+\operatorname{EC}(\partial L\cap\partial F),
    \end{align}
    and thus, by \eqref{eq:EC-cavity-excursion-face-bound}, \eqref{eq:EC-cavity-L-boundary-recursion} and induction on $k$,
    \begin{align}\label{eq:EC-cavity-L-boundary-bound}
        |\operatorname{EC}(\partial L\cap F)|\leq C_dN_v.
    \end{align}

    Combining \eqref{eq:EC-cavity-KLH-euler},
    \eqref{eq:EC-cavity-excursion-face-bound}, and
    \eqref{eq:EC-cavity-L-boundary-bound}, for
    $F\in\{B_v,\partial_{\mathrm{LL}}B_v\}$,
    \begin{align}\label{eq:EC-cavity-K-face-bound}
        \left|\operatorname{EC}(\partial K\cap F)-\operatorname{EC}(K\cap F)\right|\leq C_dN_v.
    \end{align}

    Since $\mathscr A$ is finite and its cavities are disjoint, Euler
    additivity, the definition of $w_v(D)$, and
    \eqref{eq:EC-cavity-K-face-bound} now give
    \begin{align}\label{eq:EC-cavity-finite-sum-bound}
        \left|\sum_{D\in\mathscr A}w_v(D)\right|\leq C_dN_v.
    \end{align}

    For any finite $\mathscr J\subseteq\mathscr D_{-u}$, 
    apply \eqref{eq:EC-cavity-finite-sum-bound} separately to the subfamilies on which $w_v(D)$ is nonnegative and negative, 
    then we have
    \begin{align}\label{eq:EC-cavity-finite-absolute-bound}
        \sum_{D\in\mathscr J}|w_v(D)|\leq C_dN_v.
    \end{align}

    Taking the supremum over finite $\mathscr J$ in \eqref{eq:EC-cavity-finite-absolute-bound} proves
    \begin{equation}\label{eq:EC-cavity-local-bound}
        \sum_{D\in\mathscr D_{-u}}|w_v(D)|\leq C_dN_v.
    \end{equation}

    Since \citet[Proposition~4.14]{McAuley2026a} gives $\mathbb EN_v<\infty$, \eqref{eq:EC-cavity-local-bound} implies
    \begin{align}\label{eq:EC-cavity-local-integrability}
        \mathbb E\sum_{D\in\mathscr D_{-u}}|w_v(D)|<\infty.
    \end{align}

    Stationarity and Tonelli's theorem, together with \eqref{eq:EC-cavity-local-integrability}, then give
    \begin{align}\label{eq:EC-cavity-transport-integrability}
        \mathbb E\sum_{\substack{D\in\mathscr D_{-u}\\z(D)=\mathbf 0}}\sum_{v\in\mathbb Z^d}|w_v(D)|
        &=\sum_{v\in\mathbb Z^d}\mathbb E\sum_{\substack{D\in\mathscr D_{-u}\\z(D)=\mathbf 0}}|w_v(D)|\notag\\
        &=\sum_{v\in\mathbb Z^d}\mathbb E\sum_{\substack{D\in\mathscr D_{-u}\\z(D)=-v}}|w_{\mathbf 0}(D)|\notag\\
        &=\mathbb E\sum_{D\in\mathscr D_{-u}}|w_{\mathbf 0}(D)|<\infty.
    \end{align}
    
\end{proof}

\begin{proof}[Proof of Lemma~\ref{lem:EC-cavity-topological-bound}]
    \textbf{\textit{Step 1.}}
    Let $H_i(A;\mathbb Q)$ denote the $i$th singular homology group of a space $A$.  
    Define the Betti numbers of $C$ and $D$ by
    \begin{align}\label{eq:EC-cavity-Betti-definitions}
        b_i(C):=\dim H_i(C;\mathbb Q),\qquad
        b_i(D):=\dim H_i(D;\mathbb Q),\qquad i\geq 0.
    \end{align}

    Denote $H^i(A;\mathbb Q)$ for cohomology group and $\widetilde H_i(A;\mathbb Q)$ for reduced homology group.  
    In particular, for a nonempty path-connected space $A$, $\widetilde H_0(A;\mathbb Q)=0$,
    whereas 
    \begin{align}\label{eq:EC-cavity-reduced-homology}
        \widetilde H_i(A;\mathbb Q)=H_i(A;\mathbb Q)
    \end{align}
    for $i\geq 1$.
    Apply universal coefficient theorem and Alexander duality for $C$ and $D$ {\citet[Theorem~3.2 and Corollary~3.45]{Hatcher2002}} gives
    \begin{align}\label{eq:EC-cavity-Alexander-duality}
        \operatorname{Hom}_{\mathbb Q}\bigl(H_i(A;\mathbb Q),\mathbb Q\bigr)\cong H^i(A;\mathbb Q)\cong\widetilde H_{d-i-1}(S^d\setminus A;\mathbb Q),
        \quad A\in\{C,D\},\quad 1\le i\le d-1.
    \end{align}
    
    Identify $S^d$ with the one-point compactification $\mathbb R^d\cup\{\infty\}$.  
    Let $V_0$ be the component of $S^d\setminus C$ containing $\infty$.  
    By the definition of filling, $S^d\setminus D=V_0$.
    Thus the direct-sum decomposition of homology {\citet[Proposition~2.6]{Hatcher2002}} gives
    \begin{align}\label{eq:EC-cavity-complement-homology-sum}
        H_{d-i-1}(S^d\setminus C;\mathbb Q)\cong H_{d-i-1}(V_0;\mathbb Q)\oplus H_{d-i-1}\bigl((S^d\setminus C)\setminus V_0;\mathbb Q\bigr).
    \end{align}

    Therefore for $1\leq i<d-1$, \eqref{eq:EC-cavity-reduced-homology} and \eqref{eq:EC-cavity-complement-homology-sum} gives
    \begin{align}\label{eq:EC-cavity-complement-dimension-bound}
        \dim \widetilde{H}_{d-i-1}(V_0;\mathbb{Q})=\dim H_{d-i-1}(V_0;\mathbb Q)
        \leq\dim H_{d-i-1}(S^d\setminus C;\mathbb Q)=\dim\widetilde H_{d-i-1}(S^d\setminus C;\mathbb Q).
    \end{align}

    For $i=d-1$, we instead use $\widetilde H_0(V_0;\mathbb Q)=0$, since $V_0$ is path connected.
    Consequently, \eqref{eq:EC-cavity-Alexander-duality} and \eqref{eq:EC-cavity-complement-dimension-bound} gives
    \begin{align}\label{eq:EC-cavity-Betti-comparison}
        b_i(D)=\dim\widetilde H_{d-i-1}(V_0;\mathbb Q)
        \leq\dim\widetilde H_{d-i-1}(S^d\setminus C;\mathbb Q)=b_i(C),
        \quad 1\leq i\leq d-1.
    \end{align}

    \textbf{\textit{Step 2.}}
    For a nondegenerate critical point $p$ of $g$, 
    its Morse index $\operatorname{ind}_g(p)$ is the number of negative eigenvalues of $\nabla^2g(p)$.  
    Denote the number of critical points in $C$ with Morse index $i$ as  
    \begin{align*}
        m_i(C):=\#\{p\in C:\nabla g(p)=0,\ \operatorname{ind}_g(p)=i\},\quad 0\le i\le d.
    \end{align*}

    As $\partial C=(g|_C)^{-1}(a)$ is a regular upper level and
    $(g|_C)^{-1}(b)=\varnothing$ for $b<\min_C g$, we apply the weak
    Morse inequalities for compact manifolds with boundary in
    \citet[Chapter~6, Theorems~3.5 and~3.7]{Hirsch1976} to obtain
    \begin{equation}\label{eq:EC-cavity-Morse-bound}
        b_i(C)=\dim H_i(C;\mathbb Q)\leq m_i(C),
        \quad 0\le i\le d.
    \end{equation}
    
    Moreover, by definition 
    \begin{align*}
        N_{\mathrm{bad}}(C)=\sum_{i=1}^{d}m_i(C).
    \end{align*}

    \textbf{\textit{Step 3.}}
    Since $D$ is connected, $b_0(D)=1$.
    The homological formula for the Euler characteristic, \eqref{eq:EC-cavity-Betti-comparison} and \eqref{eq:EC-cavity-Morse-bound} now yield
    \begin{align*}
        \operatorname{EC}(D)
        &=1+\sum_{i=1}^{d-1}(-1)^i b_i(D)
        \geq 1-\sum_{i=1}^{d-1}b_i(D)
        \geq 1-\sum_{i=1}^{d-1}b_i(C)
        \geq 1-\sum_{i=1}^{d-1}m_i(C)\\
        &\geq 1-N_{\mathrm{bad}}(C).
    \end{align*}
    This proves \eqref{eq:EC-cavity-topological-bound}.
\end{proof}

\begin{proof}[Proof of Lemma~\ref{lem:EC-cavity-bad-intensity}]
    Denote 
    \begin{align*}
        X=f(\mathbf0),\qquad
        G=\nabla f(\mathbf0),\qquad
        M=\nabla^2f(\mathbf0),\qquad
        \Sigma=\operatorname{Cov}(G).
    \end{align*}
 
    Let $K(x)=\mathbb E[f(x)f(\mathbf0)]$.  
    $K$ is even function, so differentiating the covariance at the origin gives, for all $i,j,k$,
    \begin{align*}
        \operatorname{Cov}(X,G_i)&=0,
        &\Sigma_{ij}=-\partial_{ij}K(\mathbf0),\\
        \operatorname{Cov}(M_{ij},X)
        &=\partial_{ij}K(\mathbf0)=-\Sigma_{ij},
        &\operatorname{Cov}(M_{ij},G_k)=0.
    \end{align*}
    
    We first verify that $\Sigma\succ0$.  
    Recall that $f=q*W$, thus for every $v\ne\mathbf0$,
    \begin{align*}
        v^{\mathsf T}\Sigma v=\operatorname{Var}(v\cdot\nabla f(\mathbf0))
        =\int_{\mathbb R^d}|v\cdot\nabla q(x)|^2\,dx\geq 0.
    \end{align*}

    If this integral were zero, 
    the continuity of $\nabla q$ would imply $v\cdot\nabla q\equiv0$, and hence $q(x+tv)=q(x)$ for all $x$ and $t$.
    Since $|x+tv|\to\infty$ as $t\to\infty$, 
    Assumption~\ref{ass:McAuley} would then give $q(x)=\lim_{t\to\infty}q(x+tv)=0$ for every $x$, 
    contradicting $\|q\|_{L^2}=1$.  
    Therefore $\Sigma\succ0$, and we denote $\lambda=\lambda_{\min}(\Sigma)>0$.

    Define the centred Gaussian symmetric matrix $Z:=M+X\Sigma$.
    Since $\operatorname{Var}(X)=1$, the preceding covariance identities give
    \begin{align*}
        \operatorname{Cov}(Z_{ij},X)=-\Sigma_{ij}+\Sigma_{ij}=0,
        \qquad
        \operatorname{Cov}(Z_{ij},G_k)=0,
    \end{align*}
    so joint Gaussianity implies that $Z$ is independent of $(X,G)$.
    Additionally, $G$ is independent of $(X,M)$.  
    Thus
    \begin{align*}
        M=-X\Sigma+Z,\qquad
        \mathcal L(M\mid X=-t,G=\mathbf0)=\mathcal L(t\Sigma+Z),
        \qquad t\ge0.
    \end{align*}

    For $u\geq 1$, denote $E_u=\{X\leq -u,\ M\not\succ0\}$.
    On $\{X\leq -u\}$, for every unit vector $v$,
    \begin{align*}
        v^{\mathsf T}Mv=(-X)v^{\mathsf T}\Sigma v+v^{\mathsf T}Zv
        \geq\lambda u-\|Z\|_{\mathrm{op}}.
    \end{align*}

    Hence under $\{X\leq -u\}$, $M\succ0$ whenever $\|Z\|_{\mathrm{op}}<\lambda u$, and
    \begin{equation}\label{eq:EC-cavity-bad-event-inclusion}
        E_u\subseteq\{X\leq -u\}\cap\{\|Z\|_{\mathrm{op}}\geq\lambda u\}.
    \end{equation}

    Since $\|Z\|_{\mathrm{op}}\leq d\max_{i,j}|Z_{ij}|$, a union bound and Gaussian tails give constants $C,c_1>0$ such that
    \begin{align*}
        \mathbb P(\|Z\|_{\mathrm{op}}\ge\lambda u)
        \leq\sum_{i,j=1}^d\mathbb P\bigl(|Z_{ij}|\ge\lambda u/d\bigr)
        \leq Ce^{-c_1u^2}.
    \end{align*}

    As $X$ is standard normal and independent of $Z$,
    \eqref{eq:EC-cavity-bad-event-inclusion} yields
    \begin{align}\label{eq:EC-cavity-bad-event-tail}
        \mathbb P(E_u)\leq\mathbb P(X\le-u)\,
        \mathbb P(\|Z\|_{\mathrm{op}}\ge\lambda u)\leq Ce^{-u^2/2}e^{-c_1u^2}
        =Ce^{-(1/2+c_1)u^2}.
    \end{align}

    By \citet[Eq.~(6.18)]{AzaisWschebor2009},
    applied to $-f$ with $k=2$, which counts the number of critical points of $f$,
    and using the stationarity of $f$, we obtain
    \begin{align}\label{eq:EC-cavity-all-critical-Kac-Rice}
        &\mathbb E\#\bigl\{x\in[0,1)^d:\nabla f(x)=\mathbf0,\ f(x)\le-u\bigr\}\notag\\
        =&\int_{[0,1)^d}\!dx\int_u^\infty\mathbb E\!\left[|\det(-M)|\mid -X=t,\ -G=\mathbf0\right]p_{-X,-G}(t,\mathbf0)\,dt\notag\\
        =&\int_{-\infty}^{-u}\mathbb E\!\left[|\det M|\mid X=t,\ G=\mathbf0\right]p_{X,G}(t,\mathbf0)\,dt\notag\\
        =&p_G(\mathbf0)\,\mathbb E\!\left[|\det M|\mathbf1_{\{X\le-u\}}\mid G=\mathbf0\right].
    \end{align}

    Here $p_{X,G}$ and $p_{-X,-G}$ denote the corresponding joint densities;
    $p_G$ denotes the density of $G$.
    The fixed level $-u$ contains no critical point almost surely.

    Similarly, applying the same formula to $-f$ with $k=1$, which counts the local maxima of $-f$, gives
    \begin{align}\label{eq:EC-cavity-local-maxima-Kac-Rice}
        &\mathbb E\#\bigl\{x\in[0,1)^d:\nabla f(x)=\mathbf0,\ f(x)\le-u,\ \nabla^2f(x)\succ0\bigr\}\notag\\
        =&p_G(\mathbf0)\,\mathbb E\!\left[|\det M|\mathbf1_{\{X\le-u\}}\mathbf1_{\{M\succ0\}}\mid G=\mathbf0\right].
    \end{align}

    By the definition of $\nu_{\mathrm{bad}}(u)$, 
    subtracting \eqref{eq:EC-cavity-local-maxima-Kac-Rice} from \eqref{eq:EC-cavity-all-critical-Kac-Rice} yields
    \begin{align}\label{eq:EC-cavity-bad-count-Kac-Rice}
        \nu_{\mathrm{bad}}(u)
        &=p_G(\mathbf0)\,\mathbb E\!\left[|\det M|\mathbf1_{\{X\le-u\}}\bigl(1-\mathbf1_{\{M\succ0\}}\bigr)\mid G=\mathbf0\right]\notag\\
        &=p_G(\mathbf0)\,\mathbb E\!\left[|\det M|\mathbf1_{E_u}\mid G=\mathbf0\right]\notag\\
        &=p_G(\mathbf0)\,\mathbb E\!\left[|\det M|\mathbf1_{E_u}\right],
    \end{align}
    where the last equality uses the independence of $G$ from $(X,M)$.

    Since $\det M$ is a polynomial in finitely many Gaussian variables,
    $\mathbb E|\det M|^p<\infty$ for every finite $p>1$.
    By H\"older's inequality and \eqref{eq:EC-cavity-bad-event-tail},
    \begin{align}\label{eq:EC-cavity-bad-count-Holder}
        \nu_{\mathrm{bad}}(u)
        &\le p_G(\mathbf0)\bigl(\mathbb E|\det M|^p\bigr)^{1/p}\mathbb P(E_u)^{1-1/p}\notag\\
        &\le C_p e^{-(1-1/p)(1/2+c_1)u^2}.
    \end{align}
    Fix $p>1+(2c_1)^{-1}$ and set $\delta=(1-1/p)(1/2+c_1)-1/2>0$.
    Thus \eqref{eq:EC-cavity-bad-count-Holder} proves \eqref{eq:EC-cavity-bad-intensity}.
\end{proof}

\begin{proof}[Proof of Lemma~\ref{lem:EC-cavity-large-tail}]
    \textbf{\textit{Step 1.}}
    Fix a sufficiently large $L$ and suppose that, for some
    $D\in\mathscr F_{-u}$,
    \begin{align}\label{eq:EC-cavity-dyadic-event}
        \mathbf 0\in D,
        \qquad L<\operatorname{diam}D\le2L.
    \end{align}

    Let $C(D)$ be the component of $\{f\le-u\}$ whose filling is $D$.
    The filling properties used above give
    \begin{align*}
        C(D)\subseteq D\subseteq\Lambda_{2L},
        \qquad \operatorname{diam}C(D)=\operatorname{diam}D>L.
    \end{align*}

    Choose $x,y\in C(D)$ with $|x-y|>L$, and choose
    $v\in\mathbb Z^d$ such that $x\in B_v$.  Then
    $v\in\Lambda_{2L+1}$ and
    \begin{align*}
        |y-v|_\infty\ge |y-x|_\infty-|x-v|_\infty
        >\frac{L}{\sqrt d}-1.
    \end{align*}

    Since $C(D)$ is path connected, for all sufficiently large $L$ this
    implies
    \begin{align}\label{eq:EC-cavity-long-connection}
        B_v\overset{\{f\le-u\}}{\longleftrightarrow}
        v+\partial\Lambda_{n_L},
        \qquad
        n_L:=\left\lfloor\frac{L}{4\sqrt d}\right\rfloor>2.
    \end{align}

    Since $\{f\le-u\}\subseteq\{f\le a_0\}$, stationarity,
    $B_{\mathbf 0}\subseteq\Lambda_1$, and $f\overset d=-f$ allow us to
    apply \citet[Theorem~1.1]{McAuley2026a} at the fixed level
    $-a_0>\ell_c$.  Thus, uniformly in $v$,
    \begin{align*}
        \mathbb P\left(B_v\overset{\{f\le-u\}}{\longleftrightarrow}v+\partial\Lambda_{n_L}\right)
        \leq\mathbb P\left(\Lambda_1\overset{\{f\ge-a_0\}}{\longleftrightarrow}\partial\Lambda_{n_L}\right)
        \leq Ce^{-cL}.
    \end{align*}

    There are at most $CL^d$ possible indices
    $v\in\mathbb Z^d\cap\Lambda_{2L+1}$.  Hence
    \begin{align}\label{eq:EC-cavity-one-scale-tail}
        \mathbb P\left(\exists D\in\mathscr F_{-u}:
        \mathbf 0\in D,\ L<\operatorname{diam}D\le2L\right)
        \le CL^de^{-cL}.
    \end{align}

    Summing \eqref{eq:EC-cavity-one-scale-tail} over $L=2^jR$, $j\ge0$ gives
    \begin{align*}
        \mathbb P\left(\exists D\in\mathscr F_{-u}:\mathbf 0\in D,\ \operatorname{diam}D>R\right)
        \leq C\sum_{j\ge0}(2^jR)^de^{-c2^jR}
        \leq C_0e^{-c_0R}.
    \end{align*}

    Increasing $C_0$ if necessary extends the estimate to every $R\ge1$.

    \textbf{\textit{Step 2.}} We prove the existence of maximal fillings.
    By stationarity and \eqref{eq:EC-cavity-large-tail}, for every $x\in\mathbb Q^d$,
    \begin{align*}
        \mathbb P\left(\exists D\in\mathscr F_{-u}:
        x\in D,\ \operatorname{diam}D>R\right)\le C_0e^{-c_0R}.
    \end{align*}

    Letting $R\to\infty$ and taking a countable intersection over $x\in\mathbb Q^d$, we obtain, almost surely,
    \begin{align*}
        R_x:=\sup\bigl(\{\operatorname{diam}D:D\in\mathscr F_{-u},\ x\in D\}\cup\{0\}\bigr)<\infty,
        \quad x\in\mathbb Q^d.
    \end{align*}

    Also, almost surely $-u$ is a regular value of $f$, and every bounded set contains only finitely many critical points of $f$; see
    \citet[Lemma~4.10 and Proposition~4.14]{McAuley2026a}.

    Fix $D\in\mathscr F_{-u}$ and choose $x\in\mathbb Q^d\cap\operatorname{int}D$.
    Every $D'\in\mathscr F_{-u}$ containing $D$ satisfies
    \begin{align*}
        C(D')\subseteq D'\subseteq\{y\in\mathbb R^d:|y-x|\le R_x\}:=B(x,R_x).
    \end{align*}

    Each compact sublevel component $C(D')$ attains its minimum at an interior critical point, 
    while distinct components give distinct critical points in $B(x,R_x)$.  
    Consequently, only finitely many members of $\mathscr F_{-u}$ contain $D$.
    A maximal member of this finite family is the maximal filling.

    \textbf{\textit{Step 3.}} Finally we prove local finiteness of $\mathscr D_{-u}$.
    Applying Lemma~\ref{lem:cubical-excursion-components} to
    $g(x)=-f(v+x)$ at level $u$, we find that
    $\{f\le-u\}\cap B_v$ has finitely many connected components.

    For $D=\operatorname{Fill}(C(D))$, closedness of $C(D)$ gives $\partial D\subseteq C(D)$.
    Thus only finitely many $D\in\mathscr D_{-u}$ satisfy $\partial D\cap B_v\ne\varnothing$.
    If $D\cap B_v\ne\varnothing$ but $\partial D\cap B_v=\varnothing$,
    connectedness of $B_v$ implies $B_v\subseteq\operatorname{int}D$.
    Since maximal fillings are disjoint, so at most one maximal filling contains $B_v$. 
    Writing $\operatorname{Comp}(A)$ for the collection of connected components of $A$, we obtain
    \begin{equation}\label{eq:EC-cavity-local-finiteness}
        \#\{D\in\mathscr D_{-u}:D\cap B_v\ne\varnothing\}
        \le 1+\#\operatorname{Comp}(\{f\le-u\}\cap B_v)<\infty.
    \end{equation}

    Every compact subset of $\mathbb R^d$ is covered by finitely many unit cubes, which proves the third result.
\end{proof}

\textbf{Funding.}
Liu S.H. was partially supported by the Fundamental Research Funds for the Central Universities DUT25RC(3)133.
Shao Q.M. was partially supported by National Nature Science Foundation of China NSFC 12031005 and Shenzhen Outstanding Talents Training Fund, China.

\end{document}